\documentclass[12pt,hidelinks]{amsart}
\usepackage{amsmath}
\usepackage{bbm}
\usepackage{enumitem}
\usepackage{amsfonts}
\usepackage{amssymb}
\usepackage{color}
\usepackage[dvipsnames]{xcolor}
\usepackage{amsthm,doi}
\usepackage{graphicx}
\usepackage{hyperref}
\usepackage{mathtools}
\hypersetup{hidelinks}

\title[Zeros of the Spectrogram of Colored Noise]{Zeros of the Spectrogram of Colored Noise}

\author[L. A. Escudero]{Luis Alberto Escudero}
\address[L. A. E.]{Acoustics Research Institute, Austrian Academy of Sciences, Dominikanerbastei 16, 1010 Vienna, Austria}
\email{luis.escudero@oeaw.ac.at}

\author[G. Koliander]{G\"{u}nther Koliander}
\address[G. K.]{Acoustics Research Institute, Austrian Academy of Sciences, Dominikanerbastei 16, 1010 Vienna, Austria}
\email{guenther.koliander@oeaw.ac.at}

\author[J. L. Romero]{Jos\'{e} Luis Romero}
\address[J. L. R.]{Faculty of Mathematics, University of Vienna, Oskar-Morgenstern-Platz 1, A-1090 Vienna, Austria, and Acoustics Research Institute, Austrian Academy of Sciences, Dominikanerbastei 16,	1010 Vienna, Austria}
\email{jose.luis.romero@univie.ac.at}

\thanks{This research was funded in whole or in part by the Austrian Science Fund (FWF): 10.55776/Y1199. For open access purposes, the authors have applied a CC BY public copyright license to any author-accepted manuscript version arising from this submission.}



\newcommand{\bZ}{\mathbb{Z}}
\newcommand{\bN}{\mathbb{N}}
\newcommand{\bR}{\mathbb{R}}
\newcommand{\bC}{\mathbb{C}}
\newcommand{\lone}[1]{L^1 \left(#1\right)}

\newcommand{\rd}[1]{\mathcal{S}(#1)}
\newcommand{\rdf}[1]{\mathcal{S}'(#1)}

\def\noise{\mathcal{W}}
\def\xvect{\boldsymbol{x}}
\def\xel{x}
\def\il{y_1}
\def\ill{y_2}
\def\mapcoord{\chi}
\def\shift{\tau}
\def\frequency{\omega}

\def\psd{S}
\newcommand{\smoothedpsd}[1]{#1_{\phi}}
\def\rhoTF{\rho_{1,\, \text{Spec}}}
\def\rhoBargmann{\rho_{1,\, \text{Barg}}}
\def\hatrhoTF{\widehat{\rho}_{1,\, \text{Spec}}}

\newcommand{\discreteSTFT}{\mathcal{V}}

\def\window{\varphi}
\newcommand{\cn}{\noise_{\psd{}}}

\DeclarePairedDelimiter\abs{\lvert}{\rvert}
\DeclarePairedDelimiter\bigabs{\big\lvert}{\big\rvert}
\DeclarePairedDelimiter\Bigabs{\Big\lvert}{\Big\rvert}
\DeclarePairedDelimiter\biggabs{\bigg\lvert}{\bigg\rvert}

\newcommand{\ip}[2]{\langle #1,\, #2\rangle}
\newcommand{\bigip}[2]{\big\langle #1,\, #2\big\rangle}
\newcommand{\E}[1]{\mathbb{E}\big[ #1 \big]}
\newcommand{\bE}{\mathbb{E}}
\newcommand{\biggE}[1]{\mathbb{E}\biggl[ #1 \biggr]}
\newcommand{\Econd}[2]{\mathbb{E}\big[ #1 \, \big\vert \, #2  \big]}
\def\numsim{n}

\def\simnoise{\vec{\mathcal{W}}}
\def\simwhite{\boldsymbol{\zeta}}
\def\gridstep{\delta}  
\def\empintegral{Z}

\def\fouriersymb{\mathcal{F}}
\newcommand{\ifft}[1]{\fouriersymb^{-1}(#1)}
\newcommand{\fft}[1]{\fouriersymb(#1)}
\newcommand{\autocorr}[1]{K_{#1}}
\newcommand{\stft}[2]{V_{#1} #2}
\newcommand{\bargmann}[1]{\mathcal{B} [#1]}
\newcommand{\fourier}[1]{\fouriersymb \{#1\}}

\newcommand{\mo}[1]{\mathcal{M}_{#1}}
\newcommand{\tr}[1]{\mathcal{T}_{#1}}
\let\Re\relax

\DeclareMathOperator{\Re}{Re}
\DeclareMathOperator{\sgn}{sgn}
\newcommand{\tfsnr}{\Gamma}

\newtheorem{thm}{Theorem}[section]
\newtheorem{prop}[thm]{Proposition}
\newtheorem{lemma}[thm]{Lemma}
\newtheorem{definition}[thm]{Definition}
\newtheorem{remark}[thm]{Remark}
\newtheorem*{remark*}{Remark}

\allowdisplaybreaks
\usepackage{bigints}
\usepackage{subcaption}
\usepackage[section]{placeins}
\usepackage{eqparbox}
\let\oldgeq\geq
\renewcommand{\geq}{\,\oldgeq\,}

\numberwithin{equation}{section}

\begin{document}\setlength{\jot}{10pt}
	
	\begin{abstract}
    We study the expected number of zeros of the Short-Time Fourier Transform (STFT) with a Gaussian window for signals degraded by complex Gaussian colored noise. We provide an exact formula for this quantity and investigate its asymptotic behavior. From a computational perspective, we propose an algorithmic method to recover a smoothed power spectral density (PSD) directly from the spatial distribution of the transform's zeros. Our results formalize the commonly-held heuristic that detection algorithms based on spectrogram zeros, though designed for white noise, also perform adequately under moderately colored noise.
	\end{abstract}
	
	\maketitle
	
	\section{Introduction and results}
	\subsection{Introduction}

The short-time Fourier transform (STFT)
of a function $f: \mathbb{R}\to\mathbb{C}$ is defined with respect to an auxiliary smooth and rapidly decaying \emph{window function} $\window: \mathbb{R}\to\mathbb{C}$ as
    \begin{equation}\label{eq:stft_def}
        \stft{\window}{f}(x, \omega) = \int_{-\infty}^{\infty} f(t) \, \overline{\window(t-x)} \, e^{-2\pi i \omega t} \, dt, \qquad x, \omega \in \bR.
    \end{equation}
    In applications, $f$ is often called a \emph{signal} and its spectrogram $\text{Spec}_f \coloneqq \abs{\stft{\window}{f}}^2$ quantifies the signal's energy distribution over the time-frequency plane. Among all possible windows, the Gaussian $\window(t) \coloneqq 2^{1/4} e^{-\pi t^2}$ holds a privileged role because it simultaneously minimizes time and frequency spread, and is the preferred choice in practice. For this choice, the STFT is sometimes referred to as the Gabor transform.
    
    Classically, time-frequency analysis has prioritized the study of high-energy regions, specifically focusing on the spectrogram's maxima as primary carriers of signal information. More recently, an alternative paradigm has revealed the rich content encoded by the \emph{zeros} of the transform, particularly when signals are submerged in noise \cite{gardner2006sparse,flandrin2015time}. Several novel signal processing strategies exploit 
    the statistical pattern of the spectrogram zeros to extract signals from noise \cite{flandrin2015time, flandrinsilence,
    bardeonzeros, bardetf,
    ghosh2022signal, miramont2024unsupervised, miramont2025filtering, pascal2024point}.

In most of the existing literature on spectrogram zeros, noise is assumed to
be stationary and strictly uncorrelated across time, a condition known as \emph{whiteness} because all noise frequencies contribute equally to a constant \emph{power spectral density}. While never strictly satisfied, these assumptions are certainly mathematically convenient. The zeros of the STFT of white noise form a \emph{stationary point process} on $\mathbb{R}^2$. Moreover, the use of the Gaussian window in \eqref{eq:stft_def} leads to a very well-studied stationary point process, namely the zero set of the so-called Gaussian Analytic Function (GAF) \cite{gafbook, whatis, bardeonzeros, bardetf}. 
    
A key feature of a random point process is its first intensity function $\rho_1$. Integrating $\rho_1$ over a Borel set yields the expected number of points therein. For the zeros of the Gabor transform of complex white noise, this first intensity is identically $1$, so the expected number of zeros in any Borel set equals its area. Consequently, deterministic signals can be detected as local anomalies as they form distinct regions that disrupt the uniform spacing of background zeros \cite{flandrin2015time, flandrinsilence, 
bardeonzeros, bardetf,
    ghosh2022signal, miramont2024unsupervised, miramont2025filtering}.

While mathematically convenient, modeling noise as strictly white is sometimes met with concern by practitioners, because mathematical methods could be overly adapted to this assumption. The goal of this article is to study the first intensity of the zeros of the Gabor transform for signals corrupted by stationary, but not necessarily white, noise and to assess the extent to which certain common heuristics are justified. Stationary noise is mainly described by its power spectral density (PSD), which is
the Fourier transform of its covariance function. By quantifying how the energy of the process is distributed across the frequency spectrum, the PSD reveals that real-world noise is typically not white but ``colored''. Examples are ubiquitous: acoustic, seismic, and physiological recordings are all characterized by PSDs that decay as an inverse power of frequency, a phenomenon often referred to as power-law noise \cite{Voss1975, bak1988self, press1978flicker, kobayashi1982, gilden19951}.

As a first contribution, we describe the first intensity of the zeros of the STFT 
of a noisy signal in terms of the PSD of the noise. Importantly, the transition from the ideal white-noise model to the more realistic colored-noise scenario alters the statistics of the zero set: the new process is stochastically invariant along the time variable but not along the frequency variable (partial stationarity); see Figure \ref{fig:intro}.

    \begin{figure}[tbh]
      \centering
      \begin{subfigure}[b]{0.48\linewidth}
          \centering
          \includegraphics[width=\linewidth]{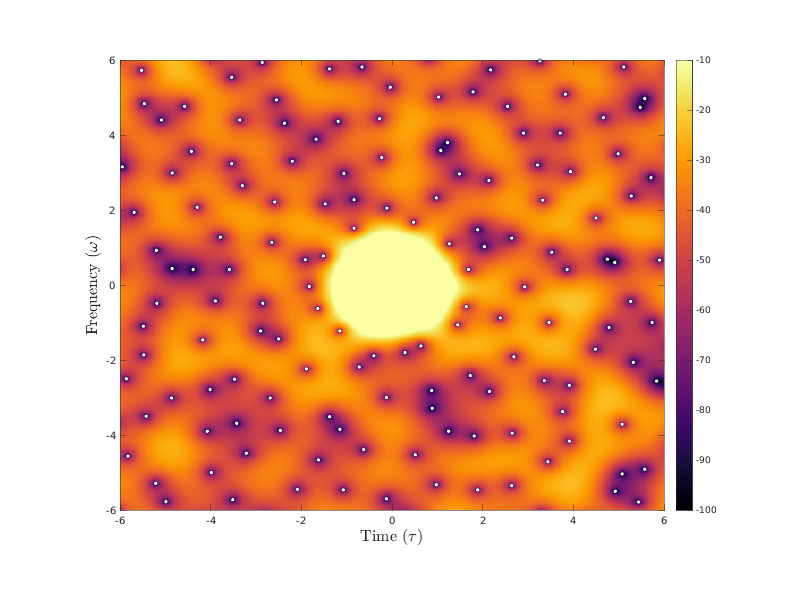}
      \end{subfigure}\hfill
      \begin{subfigure}[b]{0.48\linewidth}
          \centering
          \includegraphics[width=\linewidth]{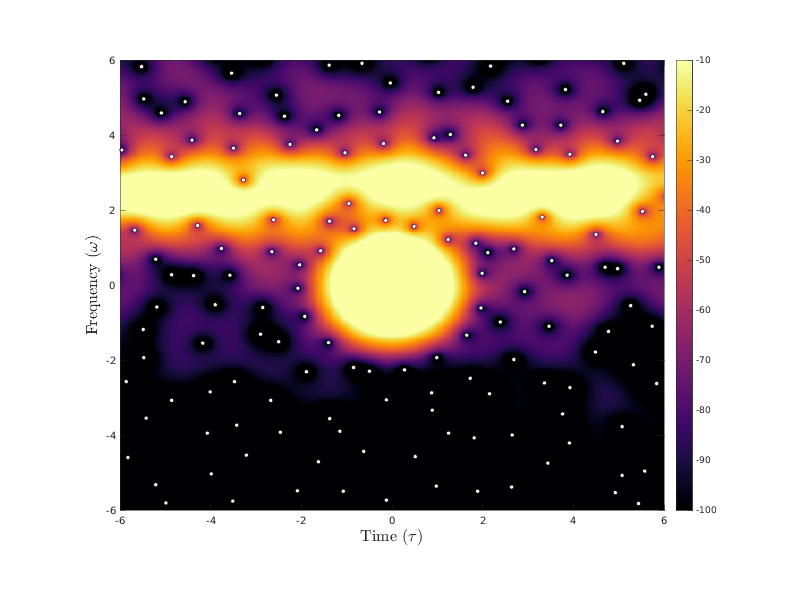}
      \end{subfigure}

      \vspace{0.5em}

      \begin{subfigure}[b]{0.48\linewidth}
          \centering
          \includegraphics[width=\linewidth]{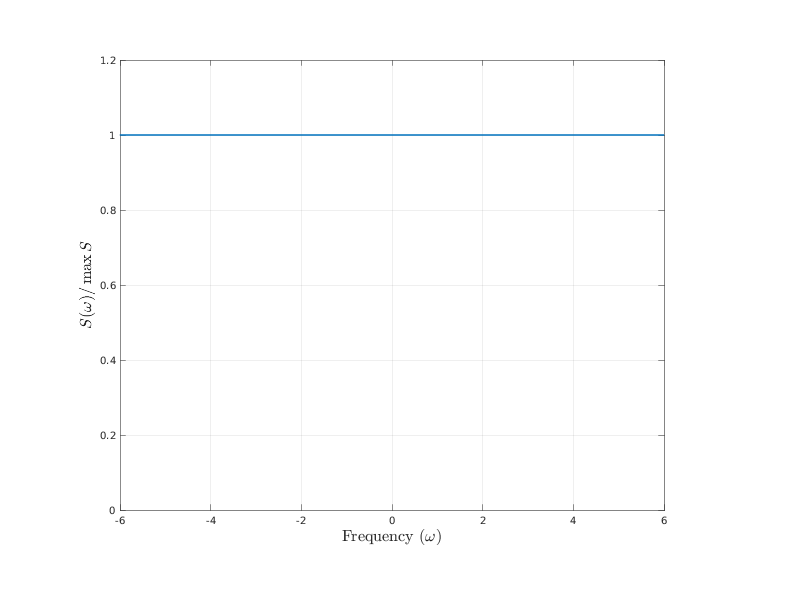}
      \end{subfigure}\hfill
      \begin{subfigure}[b]{0.48\linewidth}
          \centering
          \includegraphics[width=\linewidth]{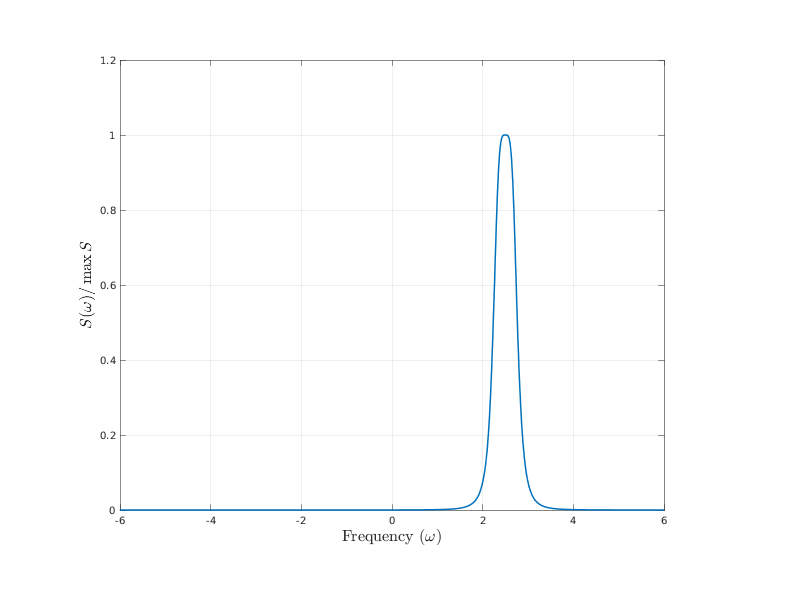}
      \end{subfigure}

\caption{Spectrogram and zeros of the Gaussian function $\exp(-\pi  t^2)$ embedded in white (left) and colored (right) noise, together with the PSD of the noise (below). In the white case, zeros populate the plane almost uniformly outside the central circular region occupied by the signal. In the colored case, the spectrogram energy is in addition concentrated around a frequency band centered at $\frequency=2.5$ (peak of PSD). Away from these regions, zeros are distributed almost uniformly.}
		\label{fig:intro}        
\end{figure}

As a second contribution, we analyze the first intensity of the STFT zeros and validate a common heuristic in the field. As formulated by Flandrin in \cite[Section 15.6]{flandrin2018explorations}, the distribution of spectrogram zeros under colored noise is expected to be similar to the one corresponding to white noise as long as the noise spectrum varies slowly,
because the effect of color is approximately that of a multiplicative factor on the spectrogram. This heuristic is important because it dispels concerns that zero-filtering techniques are overly adapted to whiteness and that they may fail under slightly more realistic setups. To formalize it, we introduce a model for slowly varying power spectra and quantify the deviation of the first intensity of STFT zeros from the white-noise baseline. We also provide  numerical validation of the analytic expressions and propose an algorithm to recover (a smoothed version of) the noise PSD directly from the observed zero set.
       
    \subsection{Main results} 
	Our first contribution is a formula for the first intensity of the zeros of the STFT for signals degraded by colored noise. Throughout the manuscript, we consider the STFT \eqref{eq:stft_def} with respect to the Gaussian window $\window(t) \coloneqq 2^{1/4} e^{-\pi t^2}$. This allows us to extend \eqref{eq:stft_def} to input signals $f$ that are tempered distributions. 
    We shall assume the following signal model:
	\begin{equation*}
		f = f_1 + \cn.
	\end{equation*} 
	Here, the deterministic component $f_1$ belongs to the \emph{modulation space} $M^\infty(\bR)$, that is, the space of all tempered distributions whose STFT is globally bounded:
	\begin{equation*}
		M^\infty(\bR) = \Big\{ f \in \rdf{\bR} : \|f\|_{M^\infty} = \sup_{\tau,\, \omega \, \in \, \mathbb{R}} |\stft{\window}{f}(\tau, \omega)| < \infty \Big\}.
	\end{equation*}
    This class contains all square-integrable functions and also all distributions commonly used in signal processing (such as Dirac deltas and periodizations and Fourier transforms thereof) \cite{benyimodulation}. The stochastic component, $\cn$, represents stationary complex Gaussian noise whose PSD $\psd$ is a measurable function of polynomial growth; see Section \ref{sec_procs} for more details.
    
    Our first result 
    concerns the spectrogram 
    \begin{equation}
    \text{Spec}_f(\tau, \omega) = |\stft{\window}{f}(\tau, \omega)|^2, \qquad (\tau, \omega) \in \mathbb{R}^2
    \end{equation}
    and describes the first order statistics of its zeros
    within a given Borel test set $\Omega \subset \mathbb{R}^2$:
    \begin{equation}\label{eq_nspec}
    	N_{\text{Spec}_f}(\Omega) = \#\big\{(\tau, \omega) \in \Omega \,:\, \text{Spec}_f(\tau, \omega) = 0\big\}.
    \end{equation}
    The formula involves two important quantities: the \emph{smoothed} PSD
\begin{equation}\label{eq_smoothed_psd}
    \smoothedpsd{S}(\frequency) \coloneqq \int_{-\infty}^{\infty} e^{-2 \pi (t - \frequency)^{2}} \psd(t) \, dt,
    \qquad \frequency \in \mathbb{R},
\end{equation}
which is the convolution of the PSD $S$ with a Gaussian function, and the \emph{time-frequency signal-to-noise-ratio}
\begin{equation}\label{eq_tfSNR}
    \tfsnr(\tau,\omega)
    =\frac{\text{Spec}_{f_1}(\tau, \omega)}{\smoothedpsd{S}(\omega)}, \qquad
    (\tau, \omega) \in \mathbb{R}^2,
\end{equation}
which measures the relation between the spectrogram and the noise at each time-frequency point. (Since $S$ is not identically zero, $\smoothedpsd{S}$ is never zero and $\tfsnr(\tau,\omega)$ is well-defined.)
	
	\begin{thm}[Expected zeros of the STFT of colored noise]\label{thm:first_intensity_stft}
	Let $f_1 \in M^\infty(\bR)$ and $\cn$ be a stationary, circularly symmetric generalized complex Gaussian process, whose PSD $\psd$ is a
    non-zero function with polynomial growth. Denote by $\smoothedpsd{S}$ the smoothed PSD \eqref{eq_smoothed_psd}.
    
    Let $f = f_1 + \cn$, consider the zero counting statistic \eqref{eq_nspec} and set    \begin{equation}\label{eq:first_point_intensity_spectrogram}
		\rhoTF(\tau, \omega) := \big[ 1 + \tfrac{1}{4\pi} (\log \smoothedpsd{S})''(\omega)+ \tfrac{1}{\sqrt{2}} 
        \widetilde{\tfsnr}(\tau, \omega)
        \big] \exp\big( -\tfrac{1}{\sqrt{2}}\tfsnr(\tau,\omega)\big), \qquad
        (\tau, \omega) \in \mathbb{R}^2,
	\end{equation}
	where
    	\begin{equation}
        \widetilde{\tfsnr}(\tau, \omega) =
\tfrac{1}{4 \pi}        
        \tfrac{\Delta\text{Spec}_{f_1}(\tau, \omega)}{ \smoothedpsd{S}(\omega)}- \tfrac{1}{2\pi}(\log \smoothedpsd{S})'(\omega) 
        \cdot
        \tfrac{\tfrac{\partial}{\partial \omega} \text{Spec}_{f_1}(\tau, \omega)}{\smoothedpsd{S}(\omega)} + \left(1+\tfrac{1}{4\pi} \cdot \big[(\log \smoothedpsd{S})'(\omega)\big]^2 \right) \cdot \tfsnr(\tau, \omega),
	\end{equation}
	and $\Delta$ is the Laplacian with respect to $(\tau, \omega)$.    
    Then, for any Borel set $\Omega \subset \mathbb{R}^2$:
	\begin{equation}\label{eq_fpi}
		\E{N_{\text{Spec}_f}(\Omega)} = \int_{\Omega} \rhoTF(\tau, \omega) \, d\tau d\omega.
	\end{equation}	
	\end{thm}
In the jargon of point-processes, \eqref{eq_fpi} means that $\rhoTF$ is the \emph{first intensity function} of the zero set of $\text{Spec}_f$. 
Theorem \ref{thm:first_intensity_stft} is proved in Section \ref{sec:zerosstft}; we make a few remarks.

$\bullet$ Theorem \ref{thm:first_intensity_stft} shows that, in time-frequency regions with high SNR, zeros are unlikely to occur, due to the exponential factor in \eqref{eq:first_point_intensity_spectrogram}. 

$\bullet$ For white noise, $S \equiv 1$, Theorem \ref{thm:first_intensity_stft} follows from 
\cite[Proposition 3.4]{efficient}. For general noise, the most interesting new element is the factor $(\log \smoothedpsd{S})'(\omega)$, which
vanishes exactly when the noise $\cn$ is white, but may have a significant impact on $\rhoTF(\tau, \omega)$ for colored noise; see Figure~\ref{fig:intro}. The fact that \eqref{eq:first_point_intensity_spectrogram} involves $\smoothedpsd{S}$ rather than $S$ is natural, since the signal $f$ is filtered with a Gaussian window.

$\bullet$ The term $\widetilde{\tfsnr}(\tau, \omega)$ is \emph{SNR-like}, as it involves ratios between derivatives of the spectrogram and the smoothed PSD.

In Sections \ref{sec:simulations} and \ref{sec:numericalexperiments} we validate \eqref{eq_fpi} numerically and illustrate some of the previous remarks. In addition, we shall exploit \eqref{eq:first_point_intensity_spectrogram} as a means to \emph{estimate} $\smoothedpsd{S}$ from the spectrogram zeros.

\bigskip

When the input signal is just noise,
$f_1\equiv 0$, the first intensity \eqref{eq:first_point_intensity_spectrogram} reduces to
\begin{equation}\label{eq_intro_rho_spec_purenoise}
    	\rhoTF(\tau, \omega) = 1 + \tfrac{1}{4\pi}(\log \smoothedpsd{S})''(\omega), \qquad (\tau, \omega) \in \bR^2.
    \end{equation}
By analyzing this expression, we can offer a formalization of the slowly-varying noise spectra heuristic \cite[Section 15.6]{flandrin2018explorations}, which asserts that 
spectrogram-zero-filtering algorithms designed for white noise can be applied more broadly. Indeed, our second result shows that for noise with mildly varying PSD, the first intensity of the STFT zeros  converges asymptotically to $1$, which is the default benchmark used when computing empirical statistics
\cite{flandrin2015time, flandrinsilence,
    bardeonzeros, bardetf,
    ghosh2022signal, miramont2024unsupervised, miramont2025filtering}.

 	\begin{thm}\label{thm:proplimit}
		Let $\cn$ be a stationary, circularly symmetric generalized complex Gaussian process with PSD $\psd \in C^2(\bR) \cap L^{\infty}(\bR) \setminus\{0\}$. Suppose that there exist $\gamma\in(0,2)$ and $\alpha,\beta>0$ such that 
        \begin{equation}\label{eq_intro_sub}
        \psd(\frequency) \geq \alpha e^{-\beta\, |\frequency|^\gamma}, \qquad \frequency\in\bR.
        \end{equation}        
		and that $\psd'(\frequency)$ and $\psd''(\frequency)$ are asymptotically dominated by ${\psd(\frequency)}$ in the following sense: for some $s_1, s_2, C_1, C_2 > 0$,
        \begin{equation}\label{eq_intro_slow}
        \abs{\psd^{(k)}(\frequency)}   \leq C_{k} (1+|\frequency|)^{-s_k} |\psd(\frequency)|, \qquad \frequency \in \bR, \quad k=1,2.
        \end{equation}
		Then, for any $\theta \in (0,1)$, the first intensity \eqref{eq_intro_rho_spec_purenoise} of the spectrogram zeros satisfies
		\begin{equation} \label{eq:result_decay_rho_spec}
			\big\lvert\rhoTF(\tau, \omega) - 1 \big\rvert 
			\leq 
			C \, 
			\|\psd\|_\infty^2 \, 
			(1+|\omega|)^{-\theta \, (1 - \frac{\gamma}{2}) \, \min(s_2, \, 2 s_1)}, \qquad \forall (\tau, \omega) \in \bR^2,
		\end{equation}
		for a constant $C=C(C_1,C_2,s_1, s_2,\alpha,\beta,\gamma, \theta)$.
 	\end{thm}
Condition \eqref{eq_intro_slow} models the notion of a slowly varying spectrum. The sub-domination condition \eqref{eq_intro_sub} precludes very fast decay of the $\psd$ and is related to the fact that we only access the signal or noise after a convolution with a Gaussian, as implemented by the STFT. Thus, if $\psd$ was allowed to have Gaussian-like decay, then the smoothed PSD $\smoothedpsd{S}$ would follow the shape of the analysis window $\phi$ rather than $S$. Indeed, with the following proposition 
we show that condition \eqref{eq_intro_sub} is sharp
for \eqref{eq:result_decay_rho_spec}: for noise colored with a Gaussian PSD, the asymptotic estimate does not hold; instead, the first intensity remains strictly below $1$.

\begin{prop}\label{prop:gaussianpsd}
Let $\psd(\frequency) = e^{-b \frequency^2}$
with $b > 0$. Then the first intensity 
\eqref{eq_intro_rho_spec_purenoise}
of the zeros of the spectrogram of 
stationary, circularly symmetric generalized complex Gaussian noise with PSD $\psd$
is $\rhoTF(\tau, \omega) = \frac{2\pi}{2\pi + b}$, $(\tau, \omega) \in \bR^2$.
	\end{prop}

     As a direct application of Theorem \ref{thm:proplimit}, we derive the following for PSDs that are ratios of polynomials, which are often encountered in practice.
    \begin{prop}\label{prop:decay_polinomial}
Let $\cn$ be a stationary, circularly symmetric generalized complex Gaussian process with
PSD $\psd(\frequency) = (P(\frequency)/Q(\frequency))^{\alpha}$, where $\alpha > 0$ and $P$ and $Q$ are polynomials taking only positive values on $\mathbb{R}$ and $\operatorname{deg}(P) < \operatorname{deg}(Q)$. 
        
        Then for any $0 < \eta < 2$ there exists a constant
$K = K(\eta,\psd)$ such that the the first intensity \eqref{eq_intro_rho_spec_purenoise} of the spectrogram zeros satisfies        
        \begin{equation}
            \abs{\rhoTF(\tau, \omega) - 1} \leq K \, (1+|\omega|)^{-\eta},\qquad (\tau, \omega)\in\bR^2.
        \end{equation}
	\end{prop}
    
    Finally, in Sections \ref{sec:simulations} and \ref{sec:numericalexperiments} we briefly consider the inverse problem of learning noise statistics from spectrogram zeros. This is done by leveraging \eqref{eq_intro_rho_spec_purenoise} to infer the smoothed PSD from empirical zero counts, exploiting the fact that time-frequency sparse signals have a moderate impact on this count at large scales. The corresponding code is openly available at \cite{code}.
    
    \subsection{Organization}
    The rest of the article is organized as follows. Section \ref{sec:preliminaries} provides notation and background on generalized random processes. In Section \ref{sec:bargmann}, we study the zeros of the Bargmann transform, which is a rescaled and weighted version of the Gabor transform that yields an analytic function. Theorem \ref{thm:first_intensity_stft} is proved in Section \ref{sec:zerosstft}. Section \ref{sec:zeromeananalysis} focuses on the asymptotic behavior of the zeros of colored noise, presenting the proofs of Theorem \ref{thm:proplimit}, its application to polynomial decay in Proposition \ref{prop:decay_polinomial}, and the calculation for Gaussian PSDs in Proposition \ref{prop:gaussianpsd}. Section \ref{sec:simulations} describes our numerical framework: the algorithmic recovery of the smoothed PSD $\smoothedpsd{S}$ from the zero set, the definition of the discrete STFT, and the experimental setup. Finally, Section \ref{sec:numericalexperiments} presents the results of our numerical experiments. 
    
\section{Preliminaries}\label{sec:preliminaries}
    
	\subsection{Notation} 
    We write $f(x) \lesssim g(x)$ if there exists a constant $C>0$ such that $f(x) \leq C g(x)$ for all $x$ in the relevant domain. Furthermore, the big-O notation $f(x) = \mathcal{O}(g(x))$ as $|x| \to \infty$ means that $|f(x)| \leq M |g(x)|$ for some constants $M$ and $x_0$, and all $|x| \geq x_0$.
    
    The Schwartz space is $\rd{\bR}$. The Fourier transform of $g\in\lone{\bR}$ is
	$\fourier{g} (\frequency) \coloneqq \int_{\bR} g(t) e^{-2\pi i t \frequency} \, dt$. The short-time Fourier transform (STFT) of $g\in L_{loc}^1(\bR)$ with respect to the window $\psi\in\rd{\bR}$ is
	$\stft{\psi}{g}(x,y) = \ip{g}{\mo{y} \tr{x} \psi}$, where $\mo{y} g (t) \coloneqq e^{2\pi i y t} g(t)$ and $\tr{x} g (t) \coloneqq g(t - x)$
    are respectively the \emph{modulation} and \emph{translation} operators. This formula extends to a tempered distribution
    $f \in \rdf{\bR}$ in the usual way. Throughout the article, there are two distinguished functions, $\window(t) \coloneqq 2^{1/4} e^{-\pi t^2}$ and $\phi(t) \coloneqq e^{-2 \pi t^2}$.

   The Wirtinger operators applied to $f:\mathbb{C}\to\mathbb{C}$ are denoted
   \[
   \partial f(z)=
   \partial_z f(z) \coloneqq \tfrac{1}{2}(\partial_x - i\partial_y)f(z), \qquad \bar\partial f(z) =
   \partial_{\bar{z}}f(z) \coloneqq \tfrac{1}{2}(\partial_x + i\partial_y)f(z), \qquad z=x+iy.\] These satisfy $\partial_{\bar{z}} \overline{f(z)} = \overline{\partial_z f(z)}$ and, for holomorphic $f$, $\partial f$ is the usual complex derivative.

    For $f\colon \bR^{2} \to \bC$ (or $f\colon \bC \to \bC$) and a subset $\Omega \subset \bR^{2}$ (resp. $\Omega \subset \bC$), we denote the number of zeros of $f$ in $\Omega$ by \[N_{f}(\Omega) \coloneqq \#\{\zeta \in \Omega \,:\, f(\zeta) = 0\}.\]
    The differential of the Lebesgue (area) measure on $
    \mathbb{R}^2 \equiv \mathbb{C}$ is denoted $dA$.
    
	\subsection{Generalized random processes, PSD and Colored Noise}\label{sec_procs}

    In the classical theory of stochastic processes, a random process $X=(X_t)_{t \in \mathbb{R}}$ is wide-sense stationary (WSS) if its mean is constant and its autocovariance depends only on the time shift $\shift$, denoted as $\autocorr{X}(\shift)$. The \emph{power spectral density} (PSD) of $X$ is defined as the Fourier transform of this autocovariance $\psd_{X} \coloneqq \fourier{\autocorr{X}}$. Under suitable regularity assumptions, one can interpret $X$ as a random function.
    In order to consider white and colored noise, we must rely on \emph{generalized random processes}, where $X$ is interpreted as a random tempered distribution. In this setting, a zero-mean process $X$ acts on test functions $\psi \in \mathcal{S}(\bR)$ as a random functional $\langle X, \psi \rangle$. For simplicity, we shall restrict ourselves to generalized Gaussian processes.

    \begin{definition}[Colored Noise and Admissible PSD]
        A measurable function $\psd\colon\bR\to[0,\infty)$ is called an \emph{admissible PSD}
        if it is not almost everywhere equal to the zero function and it has polynomial growth, that is, there exists $n>0$ such that 
        \begin{equation*}\label{eq:polygrowth}
            \sup_{\omega\in\mathbb{R}} (1+|\omega|)^{-n} |\psd(\omega)| < \infty.
        \end{equation*} 
        
        An admissible colored noise is a generalized complex stationary Gaussian process $\cn$ that is circularly symmetric and whose PSD $\psd$ is admissible. Concretely, this means that 
        $\cn: \mathcal{S}(\mathbb{R}) \to L^2(\Omega, \mathcal{F}, P)$ is a linear map that sends complex-valued test functions to complex-valued random variables on a probability space $(\Omega, \mathcal{F}, P)$ and satisfies the following conditions:
        \begin{enumerate}
		\item (\emph{Gaussianity}) For all $\phi_1, \dots, \phi_n \in \mathcal{S}(\mathbb{R})$, the complex random vector $\big(\cn(\phi_1), \dots, \cn(\phi_n)\big)$ follows a multivariate complex circularly symmetric joint Gaussian distribution.
				
		\item (\emph{Continuity}): If $\phi_k \to 0$ in $\mathcal{S}(\mathbb{R})$, then $ \mathbb{E}\left[|\cn(\phi_k)|^2\right] \to 0$.
		
		\item (\emph{Stationarity}): For all $\psi_1,\psi_2 \in \mathcal{S}(\bR)$:\begin{equation}\label{eq:covariance_noise}
            \E{\langle\cn,\psi_1\rangle\overline{\langle\cn,\psi_2\rangle}} = \int_{\bR} \fourier{\psi_2}(\frequency) \overline{\fourier{\psi_1}(\frequency)} \psd(\frequency) \, d\frequency.
        \end{equation}  
	\end{enumerate}        
        \end{definition}
   Note that the integral in \eqref{eq:covariance_noise} is convergent because $\psd$ has polynomial growth. The circular symmetry assumption entails that the process has zero mean, $\E{\langle\cn,\psi_1\rangle}=0$, and its pseudo-covariance functional vanishes identically, $\E{\langle\cn,\psi_1\rangle\langle\cn,\psi_2\rangle} = 0$.

    \begin{remark*}
    Almost every realization of an admissible colored noise $\cn$ is a tempered distribution \cite[Chapters 3 and 4]{gelfand} and therefore $\stft{\window}{\cn}$, its STFT with respect to the Gaussian window, is a well-defined smooth complex-valued function on $\mathbb{R}^2$ \cite[Chapter 1]{benyimodulation}.
    \end{remark*}
     
	\section{Bargmann Transform of Colored Noise and its Zeros}\label{sec:bargmann}
    To translate the covariance structure of colored noise into the spatial distribution of the zeros of its STFT, we rely on the Kac-Rice formula for Gaussian random fields \cite[Theorem 6.2]{level}, \cite[Chapter 11]{adler}. Under suitable regularity and non-degeneracy conditions, this theorem establishes that the expected number of zeros of a random field in $Z:\bR^n\to\bR^n$, to be found in a Borel set $E \subseteq \bR^n$ is given by integrating its first intensity function:
	\begin{equation}\label{eq:exp_zeros}
		\rho_1(t) \coloneqq \Econd{|\operatorname{Jac} Z(t)|}{Z(t)=0}\, p_{Z(t)}(0),
	\end{equation}
	where $\operatorname{Jac} Z(t)$ is the Jacobian determinant and $p_{Z(t)}$ is the probability density function of $Z(t)$. 
	
	While our ultimate goal is to describe the zeros of the STFT, computing this Jacobian determinant directly is cumbersome. We bypass this difficulty by mapping the problem to the Bargmann transform, defined as a rescaled and exponentially weighted Gabor transform. Because this transform is analytic, its Jacobian simplifies to the squared modulus of its complex derivative. Moreover, since the transformation merely rescales the zero set, the resulting first intensity can be easily mapped back to that of the original STFT.
	
    \begin{definition}[Bargmann Transform of a distribution]
		  Let $\Psi \in \rdf{\bR}$. Let $\window(t) = 2^{1/4} e^{-\pi t^2}$. The Bargmann transform of $\Psi \in \rdf{\bR}$ is the function 
		$\bargmann{\Psi}: \mathbb{C}\to\mathbb{C}$ given by
         \begin{align}\label{eq:def_bargmann}
			\bargmann{\Psi} (z)
			& \coloneqq e^{\frac{\abs{z}^2}{2}}\, e^{- i x y} \, \bigip{\Psi}{\mo{-y/\sqrt{\pi}} \, \tr{x/\sqrt{\pi}} \, \window}, \qquad z=x+iy \in \mathbb{C}.
		\end{align}
    See \cite{bargmann1961}, \cite[Chapter 1]{folland89}, \cite[Chapter 11.2]{ftfa_gro} for a detailed exposition.
	\end{definition}

    The STFT with a Gaussian window $\window(t) = 2^{1/4} e^{-\pi t^2}$ can be rewritten in terms of the Bargmann transform. Specifically, writing $z=x+iy$, we have the relation:
    \begin{align}
        e^{-\abs{z}^2/2} \,\, e^{i x y} \,\, \bargmann{f} (z)	
        &=	\ip{f}{\mo{-y/\sqrt{\pi}} \, \tr{x/\sqrt{\pi}} \, \window} = \stft{\window}{f}(x/\sqrt{\pi}, \, -y/\sqrt{\pi}). \label{eq:bargstft}
    \end{align}
    We shall also make use of the \emph{covariant derivative} of a holomorphic function $g\colon\bC\to\bC$:
    \begin{equation}\label{eq_dbarstar}
    \bar{\partial}^{*} g (z) \coloneqq \partial_{z} g(z) - \bar{z} g(z).    
    \end{equation}
    The operator $\bar{\partial}^{*}$ is the adjoint of the derivative $\partial$ with respect to $L^2$ of the Gaussian weight and is instrumental to study sampling and interpolation problems with multiplicities \cite{bs93, esharo21}.
	
	\subsection{Intensity of zeros of the Bargmann transform of colored noise} \label{sec:bargmanntransform}
    The goal of this section is to derive the following result.	\begin{thm}\label{thm:first_intensity_bargmann}
        Let $f_1 \in M^\infty(\bR)$ and $\phi(t) = e^{-2 \pi t^2}$. Let $\cn$ be admissible colored noise with PSD $S$. Let $f = f_1 + \cn$, and $\Omega \subseteq \bC$ be a Borel set. Write $z=x + iy$,
        $\omega = -y/\sqrt{\pi}$ and $F_1(z) = \bargmann{f_1}(z)$. Let $\smoothedpsd{S} \coloneqq \phi * \psd$ be the smoothed PSD and $\sigma(z) \coloneqq 2^{1/4} \, e^{\abs{z}^2/2} \sqrt{\smoothedpsd{S}(\omega)}$ the 
        \emph{local standard deviation} of $\bargmann{f}$. Then,
        \begin{equation*}
            \E{\#\{\zeta \in \Omega \,:\, \bargmann{f}(\zeta)=0\}} = \int_{\Omega} \rhoBargmann(z) \, dA(z),
        \end{equation*}
        where
        \begin{equation}
        \label{eq:rho1}
            \rhoBargmann(z) 
            \coloneqq
            \bigg(
                \frac{1}{\pi} 
                + \frac{(\log \smoothedpsd{S})''(\omega)}{4\pi^2}
                +
                \frac{1}{\pi \sigma^2(z)}
                \Bigabs{
                    \bar{\partial}^{*} F_1 (z) - \frac{i}{2\sqrt{\pi}} (\log \smoothedpsd{S})'(\omega) \, F_1(z)
                }^{2}
            \bigg)
            \exp\bigg( - \frac{\abs{F_1(z)}^2}{\sigma^2(z)} \bigg)
        \end{equation}
    and we used the  covariant derivative \eqref{eq_dbarstar}.
    \end{thm}
For the remainder of the section, we assume that $f$, $f_1$ and $\cn$ are as in Theorem \ref{thm:first_intensity_bargmann}. As a first step towards the proof of Theorem~\ref{thm:first_intensity_bargmann}, we shall compute the covariance structure of the random vector $(F_0, \partial F_0)$, where $F_0 = \bargmann{\cn}$. 
   	
	\begin{prop}\label{prop:variancebargmann}
		Let $z=x+iy$, $w=u+iv$ and $F_0 = \bargmann{\cn}$. Then
		\begin{equation}\label{eq:covariance_bargmann}
			\begin{aligned}
				\E{F_0 (z) \overline{F_0 (w)} } 
                &= 
                \sqrt{2} \, e^{\frac{1}{2} ( z^2 + \bar{w}^2)} \int_{-\infty}^{\infty} e^{-2\pi t^2} \, e^{2i\sqrt{\pi}t(z-\bar{w})} \psd(t) \, dt.
			\end{aligned}
		\end{equation}
	\end{prop}
	
	\begin{proof}
		Evaluating $\bargmann{\cn}$ and expanding the expectation, we have:
		\begin{align*}
			\E{F_0 (z) \overline{F_0 (w)}} 
			&= \bE\Big[e^{\frac{\abs{z}^2}{2}}\, e^{- i x y} \, \ip{\cn}{\mo{-y/\sqrt{\pi}} \, \tr{x/\sqrt{\pi}} \, \window} \,
				\overline{e^{\frac{\abs{w}^2}{2}}\, e^{- i u v} \, \ip{\cn}{\mo{-v/\sqrt{\pi}} \tr{u/\sqrt{\pi}} \window}}\Big] \\
			&= e^{\frac{\abs{z}^2+\abs{w}^2}{2}} e^{- i (x y-uv)} \, \E{ \ip{\cn}{\mo{-y/\sqrt{\pi}} \, \tr{x/\sqrt{\pi}} \, \window} \,
				\overline{\ip{\cn}{\mo{-v/\sqrt{\pi}} \tr{u/\sqrt{\pi}} \window}}}.
		\end{align*}
		Using \eqref{eq:covariance_noise}, this becomes:
		\begin{align*}
			\E{F_0 (z) \overline{F_0 (w)}} 
			&= e^{\frac{\abs{z}^2+\abs{w}^2}{2}} e^{- i (x y-uv)} \\
			&\qquad \cdot \int_{\bR} \fourier{\mo{-v/\sqrt{\pi}} \, \tr{u/\sqrt{\pi}} \, \window}(t)\, \overline{\fourier{\mo{-y/\sqrt{\pi}} \, \tr{x/\sqrt{\pi}} \, \window}(t)} \, \psd(t)\, dt.
		\end{align*}
		Recall that for $\alpha, \beta \in \bR$ and any function $\xi$, the Fourier transform alters the time-frequency shifts as follows $\fourier{\mo{\alpha}\, \tr{\beta} \, \xi} = \tr{\alpha}\, \mo{-\beta}\, \fourier{\xi}$. Furthermore, since $\fourier{\window} = \window$, we obtain:
		\begin{align*}
			\E{F_0 (z) \overline{F_0 (w)}} 
			&= e^{\frac{\abs{z}^2+\abs{w}^2}{2}} e^{- i (x y-uv)} \\
			&\qquad \cdot \int_{\bR} (\tr{-v/\sqrt{\pi}} \, \mo{-u/\sqrt{\pi}} \, \window (t)) \overline{(\tr{-y/\sqrt{\pi}} \, \mo{-x/\sqrt{\pi}} \, \window(t))} \, \psd(t)\, dt.
		\end{align*}
		Substituting the explicit form of $\window$, we obtain:
		\begin{align*}
			 \E{F_0 (z) \overline{F_0 (w)}} = \sqrt{2} e^{\frac{\abs{z}^2+\abs{w}^2}{2}} e^{- i (x y-uv)} \cdot \int_{\bR} e^{-2 \pi i \bigl( \frac{u}{\sqrt{\pi}}(t+\frac{v}{\sqrt{\pi}}) - \frac{x}{\sqrt{\pi}}(t+\frac{y}{\sqrt{\pi}})\bigr)} 
			e^{-\pi \bigl((t+\frac{v}{\sqrt{\pi}})^2+(t+\frac{y}{\sqrt{\pi}})^2\bigr)} \psd(t)\, dt.
		\end{align*}       
       Expanding the exponents and grouping the terms that depend on $t$ yields the result.
	\end{proof}
	
	\begin{prop}[Computation of the covariance matrix] Write $z=x+iy$, let $\omega = -y/\sqrt{\pi}$ and recall that $\smoothedpsd{S} = S * \phi$, where $\phi(t) = e^{-2 \pi t^{2}}$.
		
    The covariance structure of the joint random vector $(F_0, \partial F_0)$ can be expressed in terms of the smoothed PSD and its derivatives:
    \begin{equation}\label{eq:covariances}
        \begin{aligned}
            \E{F_0 (z)\, \overline{F_0 (z)}} 
            &= \sqrt{2}\, e^{\abs{z}^2} \smoothedpsd{S}(\omega), \\
            \E{\partial_{z} F_0 (z)\, \overline{F_0 (z)}} 
            &= \sqrt{2}\, e^{|z|^2} \big( \bar{z} \, \smoothedpsd{S}(\omega) + \tfrac{i}{2\sqrt{\pi}} \, \smoothedpsd{S}'(\omega) \big), \\
            \E{\partial_{z} F_0 (z)\, \overline{\partial_z F_0 (z)}} 
            &= \sqrt{2}\, e^{|z|^2}\, \bigl( \tfrac{1}{4\pi}\, \smoothedpsd{S}''(\omega) + \omega\, \smoothedpsd{S}'(\omega) + (1 + |z|^2)\, \smoothedpsd{S}(\omega) \bigr).
        \end{aligned}
    \end{equation}
	\end{prop}
	\begin{proof}
		Note that the smoothed PSD $\smoothedpsd{S}(\omega)$ is well-defined and smooth, as the polynomial growth of $\psd$ allows us to differentiate under the integral sign. Taking derivatives with respect to $\omega$ yields:
		\begin{align}
			\smoothedpsd{S}'(\omega) &= \int_{-\infty}^{\infty} 4\pi(t-\omega) e^{-2\pi(t-\omega)^2} \psd(t) \, dt, \label{eq:deriv1_smoothed} \\
			\smoothedpsd{S}''(\omega) &= \int_{-\infty}^{\infty} \big(16\pi^2(t-\omega)^2 - 4\pi\big) e^{-2\pi(t-\omega)^2} \psd(t) \, dt. \label{eq:deriv2_smoothed}
		\end{align}
		
		The expression for $\E{F_0 (z)\, \overline{F_0 (z)}}$ follows immediately from Proposition~\ref{prop:variancebargmann} by setting $w=z$, completing squares and considering $y = -\sqrt{\pi}\omega$. 
        
        For the next term involving a derivative, we exchange the order of differentiation and the expectation, namely $\E{\partial_{z} F_0 (z)\, \overline{ F_0 (w)}} = \partial_{z} \big(\E{F_0 (z)\, \overline{F_0 (w)}}\big)$. This is allowed because the covariance of $F_0$ is smooth and $F_0$ is almost surely continuous (and thus satisfies the so-called \emph{separability condition}) \cite[Chapter 1]{level}.         
        
		The derivative operator brings down a factor of $z + 2i\sqrt{\pi}t$. We center this factor around $\omega$ by writing $t = (t-\omega) + \omega$, which yields $z + 2i\sqrt{\pi}t = z + 2i\sqrt{\pi}\omega + 2i\sqrt{\pi}(t-\omega)$.
		Since $y = -\sqrt{\pi}\omega$, we have $z + 2i\sqrt{\pi}\omega = x + iy - 2iy = \bar{z}$. We can then integrate this expanded factor against the PSD:
		\begin{align*}
			\E{\partial_{z} F_0 (z)\, \overline{F_0 (z)}}
			&= \sqrt{2} e^{|z|^2} \biggl( \bar{z} \smoothedpsd{S}(\omega) + \frac{i}{2\sqrt{\pi}} \int_{-\infty}^{\infty} 4\pi(t-\omega) e^{-2\pi (t-\omega)^2} \psd(t) \, dt \biggr).
		\end{align*}
		Applying identity \eqref{eq:deriv1_smoothed}, we immediately obtain:
		\begin{equation}
			\E{\partial_{z} F_0 (z)\, \overline{F_0 (z)}} = \sqrt{2} e^{|z|^2} \big( \bar{z} \smoothedpsd{S}(\omega) + \tfrac{i}{2\sqrt{\pi}} \smoothedpsd{S}'(\omega) \big).
		\end{equation}
		
		To compute the variance of $\partial_z F_0(z)$, we again interchange differentiation and expectation via dominated convergence, evaluating the mixed derivative at $w=z$:
		\begin{equation*}
			\E{\partial_{z} F_0 (z)\, \overline{ \partial_z F_0 (z)}} = \partial_{\overline{w}}\, \partial_{z} \bigl(\E{F_0 (z)\, \overline{F_0 (w)}}\bigr) \big|_{w=z}.
		\end{equation*}
		After taking derivatives, we obtain:
        \begin{equation*}
            \frac{\partial^2}{\partial z \partial \bar{w}} \E{F_0 (z) \overline{F_0 (w)} } = 
            \sqrt{2} \, e^{\frac{1}{2} ( z^2 + \bar{w}^2)} \int_{-\infty}^{\infty} (z + 2i\sqrt{\pi}t)(\bar{w} - 2i\sqrt{\pi}t) \, e^{-2\pi t^2} \, e^{2i\sqrt{\pi}t(z-\bar{w})} \psd(t) \, dt.
        \end{equation*}
        Evaluating this expression at the diagonal $w=z$, the polynomial factor inside the integral becomes $\abs{z + 2i\sqrt{\pi}t}^2$. By the centering trick from the previous step $z + 2i\sqrt{\pi}t = \bar{z} + 2i\sqrt{\pi}(t-\omega)$, and we expand the squared modulus:
		\begin{align*}
			\abs{\bar{z} + 2i\sqrt{\pi}(t-\omega)}^2 
			&= |z|^2 - 4\sqrt{\pi}(t-\omega)\Re(i\bar{z}) + 4\pi(t-\omega)^2.
		\end{align*}
		Noting that $\Re(i\bar{z}) = y = -\sqrt{\pi}\omega$, the linear term simplifies to $4\pi\omega(t-\omega)$. We integrate this quadratic polynomial against the PSD:
		\begin{equation*}
			\E{|\partial_z F_0(z)|^2} = \sqrt{2} e^{|z|^2} \int_{-\infty}^{\infty} \bigl( |z|^2 + 4\pi\omega(t-\omega) + 4\pi(t-\omega)^2 \bigr) e^{-2\pi(t - \omega)^2} \psd(t) \, dt.
		\end{equation*}
		Substituting the identities \eqref{eq:deriv1_smoothed} and \eqref{eq:deriv2_smoothed} into the corresponding terms yields:
		\begin{align*}
			\E{|\partial_z F_0(z)|^2} 
			&= \sqrt{2} e^{|z|^2} \bigl( \tfrac{1}{4\pi} \smoothedpsd{S}''(\omega) + \omega \smoothedpsd{S}'(\omega) + (1 + |z|^2) \smoothedpsd{S}(\omega) \bigr).
		\end{align*}
	\end{proof}
	\begin{prop}\label{prop:kacrice_conditions} 
    $\bargmann{f}$ is a Gaussian random function, which is almost surely entire. In addition,
		\begin{enumerate}[label=(\alph*)]
            \item\label{item_a} $\bargmann{\cn}(z)$ is circularly symmetric for all $z \in \mathbb{C}$,
			\item\label{variance} $\mathrm{Var}[
            \bargmann{\cn}(z)]=
            \mathrm{Var}[\bargmann{f}(z)]>0$ for all $z \in \mathbb{C}$,
			\item\label{multiplezeros} Almost surely:
            $\partial_z \bargmann{f}(z) \not =0$ for all $z \in \mathbb{C}$ such that $\bargmann{f}(z)=0$.
            In particular, the zeros of $\bargmann{f}$ are almost surely discrete and simple.
		\end{enumerate}
	\end{prop}
	\begin{proof}
    Gaussianity and part \ref{item_a} are clear from \eqref{eq:def_bargmann} since $\ip{\cn}{\mo{-y/\sqrt{\pi}} \, \tr{x/\sqrt{\pi}} \, \window}$ is Gaussian and circularly symmetric for each $(x,y)\in\mathbb{R}^2$.
    As mentioned, $f$ is almost-surely a tempered distribution, and therefore $\bargmann{f}$ is a well-defined entire function \cite{bargmann1961}, \cite[Chapter 1]{folland89}, \cite[Chapter 11.2]{ftfa_gro}.

Let us verify \ref{variance}. Since $\psd$ is non-trivial, there exists a set of positive measure $\Omega$ where $\psd$ is strictly positive. Expanding the definition of the covariance and restricting the convolution to $\Omega$ yields a strictly positive lower bound for the variance of $\bargmann{f}$:
		\begin{align}
			\mathrm{Var}[\bargmann{f}(z)]=\E{\bargmann{\cn} (z)\, \overline{\bargmann{\cn} (z)}} 
			&\geq \sqrt{2}
			e^{\abs{z}^2}
			\int_{\Omega} e^{-2\pi (t \, + \, y/\sqrt{\pi})^2} \psd(t) \, dt \eqqcolon C_z > 0. \label{eq:lower_bound_variance}
		\end{align}
    Note that for $z$ in a given compact set, the last estimate can be made uniform. For \ref{multiplezeros}, we apply \cite[Proposition 6.5]{level}, which implies \ref{multiplezeros} provided that the probability density function of $\bargmann{f}(z)$ is bounded near $0$, locally uniformly in $z$ (that is, uniformly for $z$ in any given compact set).
    
    Recall that $\bargmann{f}(z)$ is a Gaussian vector with mean $\bargmann{f_1}(z)$, which is locally bounded by analyticity, while its variance is uniformly positive for $z$ in a given compact set, cf. \eqref{eq:lower_bound_variance}. This gives the desired locally uniform upper bound for the probability density of $\bargmann{f}(z)$ evaluated near $0$.
	\end{proof}
    \begin{remark*}
    Identifying the circularly symmetric complex random variable $\bargmann{\cn}(z)$ with a random vector $X(z)$ on $\mathbb{R}^2$, condition \ref{variance} in the last proposition means that the covariance of $X$ is positive-definite for all $z$. Similarly, \ref{multiplezeros} means that the vector on $\mathbb{R}^2$ corresponding to $\bargmann{f}(z)$ has a full-rank differential matrix at every $z$ where $\bargmann{f}(z)=0$.
    \end{remark*}
	
	We can now proceed to compute the expected number of zeros of the Bargmann transform of colored noise. 

    \begin{proof}[Proof of Theorem \ref{thm:first_intensity_bargmann}]
        Write $F=\bargmann{f_1 + \cn}$, $F_1 = \bargmann{f_1}$, and $F_0 = \bargmann{\cn}$. Since the conditions of Proposition \ref{prop:kacrice_conditions} hold, Kac-Rice's Theorem states that the expected number of zeros of $F$ in a Borel set $\Omega\subseteq\bC$ is:
        \begin{equation*}
            \E{\#\{\zeta \in \Omega: F(\zeta)=0\}} = \int_{\Omega} \Econd{|\operatorname{det} D F(z)|}{F(z)=0}\, \, p_{F(z)}(0) \, dA(z).
        \end{equation*}
        Since $F$ is analytic, $\abs{\det D F(z)} = \abs{\partial_{z} F(z)}^{2}$.
    
        On the other hand, the vector $(F(z), \partial_z F(z))$ has mean $(F_1(z),\, \partial_z F_1(z))$. Applying the Gaussian regression procedure, see, e. g., \cite[Proposition 1.2]{level}, we have
        \begin{equation*}
            \Econd{\abs{\partial_{z} F(z)}^{2}}{F(z)=0} = \E{\abs{W(z)}^2},
        \end{equation*}
        where $W(z)$ is a random variable having a complex normal distribution with mean
        \begin{equation*}
            \mu(z) \coloneqq \partial_z F_1 (z) - \frac{ \E{\partial_{z} F_0 (z)\, \overline{F_0 (z)}} }{\E{|F_0 (z)|^2} } F_1 (z),
        \end{equation*}
        and variance
        \begin{equation*}
            \Lambda (z) \coloneqq \E{|\partial_{z} F_0 (z)|^2} - \frac{\abs{\E{F_0 (z)\, \overline{\partial_{z} F_0 (z)}}}^2 }{\E{|F_0 (z)|^2}},
        \end{equation*}
        which implies that $\E{\abs{W(z)}^2} = \Lambda(z) + |\mu(z)|^2$.
        
        Substituting the covariance expressions from \eqref{eq:covariances} into $\Lambda(z)$ yields:
        \begin{equation}\label{eq:lambda}
            \Lambda(z) 
            = \sqrt{2} e^{|z|^2} 
            \bigl(
            \bigl( \tfrac{1}{4\pi} \smoothedpsd{S}''(\omega) + \omega \smoothedpsd{S}'(\omega) + (1+|z|^2) \smoothedpsd{S}(\omega) \bigr) 
            - (\smoothedpsd{S}(\omega))^{-1} \bigl| \bar{z} \smoothedpsd{S}(\omega) + \tfrac{i}{2\sqrt{\pi}} \smoothedpsd{S}'(\omega) \bigr|^2
            \bigr).
        \end{equation}
        Recalling that $z = x+iy$ and $y = -\sqrt{\pi}\omega$, we obtain:
        \begin{equation*}
            \frac{\big| \bar{z} \smoothedpsd{S}(\omega) + \frac{i}{2\sqrt{\pi}} \smoothedpsd{S}'(\omega) \big|^2}{\smoothedpsd{S}(\omega)} = |z|^2 \smoothedpsd{S}(\omega) + \omega \smoothedpsd{S}'(\omega) + \frac{1}{4\pi} \frac{(\smoothedpsd{S}'(\omega))^2}{\smoothedpsd{S}(\omega)}.
        \end{equation*}
        Replacing the above equation into \eqref{eq:lambda}:
        \begin{equation*}
            \Lambda(z) = \sqrt{2} e^{|z|^2} \smoothedpsd{S}(\omega) \biggl( 1 + \frac{1}{4\pi} \frac{\smoothedpsd{S}''(\omega) \smoothedpsd{S}(\omega) - (\smoothedpsd{S}'(\omega))^2}{(\smoothedpsd{S}(\omega))^2} \biggr) = \sigma^2(z) \biggl( 1 + \frac{1}{4\pi} (\log \smoothedpsd{S})''(\omega) \biggr).
        \end{equation*}
    
        Similarly, using \eqref{eq:covariances}, the mean term $\mu(z)$ becomes:
        \begin{align*}
            \mu(z) &= \partial_z F_1(z) - \biggl( \bar{z} + \frac{i}{2\sqrt{\pi}} \frac{\smoothedpsd{S}'(\omega)}{\smoothedpsd{S}(\omega)} \biggr) F_1(z), \\
            &= \bar{\partial}^* F_1(z) - \frac{i}{2\sqrt{\pi}} (\log \smoothedpsd{S})'(\omega) F_1(z).
        \end{align*}
        
        Noting that $p_{F(z)}(0) = (\pi \sigma^2(z))^{-1} \exp(-|F_1(z)|^2 / \sigma^2(z))$, we combine the variance and squared mean to obtain the intensity:
        \begin{align*}
            \rhoBargmann(z) 
            &= \frac{\Lambda(z) + |\mu(z)|^2}{\pi \sigma^2(z)} \exp\biggl( - \frac{|F_1(z)|^2}{\sigma^2(z)} \biggr)\\            
            &= \biggl( \frac{1}{\pi} 
            + \frac{(\log \smoothedpsd{S})''(\omega)}{4\pi^2}
            + \frac{1}{\pi \sigma^2(z)} \Bigabs{\bar{\partial}^* F_1(z) - \frac{i}{2\sqrt{\pi}} (\log \smoothedpsd{S})'(\omega) F_1(z)}^2 \biggr) \, \exp\biggl( - \frac{|F_1(z)|^2}{\sigma^2(z)} \biggr).
        \end{align*}
    \end{proof}
	
	\subsection{Zeros of the STFT of colored noise}\label{sec:zerosstft}

	As noted in \eqref{eq:bargstft}, the zeros of the STFT can be mapped to those of the Bargmann transform. Therefore, Theorem \ref{thm:first_intensity_stft} is a consequence of Theorem \ref{thm:first_intensity_bargmann} and the application of a change of variables.
	
	\begin{proof}[Proof of Theorem \ref{thm:first_intensity_stft}]
	The zeros of $\text{Spec}_f$ and $\mathcal{B}[f]$ correspond one-to-one under the mapping $\Psi(\tau, \omega) = (\sqrt{\pi}\tau, -\sqrt{\pi}\omega)$. Thus, applying the change of variables and noting that the Jacobian determinant is $|\det J_{\Psi}| = \pi$ yields:
    \begin{equation}\label{eq:rhotfae}
        \rhoTF(\tau, \omega) = \pi \, \rhoBargmann(\sqrt{\pi}\tau - i \sqrt{\pi}\omega).
    \end{equation}
    On the other hand, under the scaling $\tau = x/\sqrt{\pi}$ and $\omega = -y/\sqrt{\pi}$, we have $\bigl(\frac{\partial}{\partial x}, \frac{\partial}{\partial y} \bigr) = \frac{1}{\sqrt{\pi}} \bigl(\frac{\partial}{\partial \tau}, -\frac{\partial}{\partial \omega}\bigr)$ and $\Delta_{x,y} = \frac{1}{\pi} \Delta_{\tau, \omega}$.

	Now, write $G(x,y) \coloneqq e^{-|z|^2}|F_1(z)|^2 = \text{Spec}_{f_1}\bigl(x/\sqrt{\pi},\, -y/\sqrt{\pi}\bigr)$, where $F_1(z) = \mathcal{B}[f_1](z)$.
    The term $\widetilde{G}(x,y) \coloneqq e^{-|z|^2} \big| \bar{\partial}^* F_1(z) - \frac{i}{2\sqrt{\pi}} (\log \smoothedpsd{S})'(\omega) F_1(z) \big|^2$ governs the contribution of the deterministic signal in Theorem \ref{thm:first_intensity_bargmann}. 
 	Expanding the squared modulus in $\widetilde{G}$ we relate it to $G$. In particular, using the identity $\Delta_{x,y} = 4 \partial_z \partial_{\bar{z}}$, it follows that $e^{-|z|^2} |\bar{\partial}^* F_1(z)|^2 = \frac{1}{4}\Delta_{x,y} G(x,y) + G(x,y)$. On the other hand,
    \begin{equation*}
        2 e^{-|z|^2} \Re\left( \bar{\partial}^* F_1(z) \overline{\left( -\frac{i}{2\sqrt{\pi}} (\log \smoothedpsd{S})'(\omega) F_1(z) \right)} \right) = \frac{1}{2\sqrt{\pi}} (\log \smoothedpsd{S})'(\omega) \frac{\partial}{\partial y} G(x,y).
   \end{equation*}
    Therefore, we obtain $\widetilde{G}(x,y) = G(x,y) + \mathcal{D}_S(x,y)$, where:
	\begin{equation*}
		\mathcal{D}_S(x,y) 
        =
        \frac{1}{4}\Delta_{x,y} G(x,y) + \frac{1}{2\sqrt{\pi}}(\log \smoothedpsd{S})'(\omega) \frac{\partial}{\partial y} G(x,y) + \frac{1}{4\pi}\big((\log \smoothedpsd{S})'(\omega)\big)^2 G(x,y).
	\end{equation*}
	
    The formula for $\rhoTF(\tau, \omega)$ follows then from \eqref{eq:rhotfae} followed by the replacements derived in the previous two paragraphs. 
	\end{proof}

	\section{Asymptotic behavior of the spectrogram zeros of colored noise}\label{sec:zeromeananalysis}

    Having established the formula for the first intensity for a generic signal embedded in colored noise, we now isolate the statistics of the background noise. We then consider the pure noise model where the deterministic signal is absent ($f_1=0$). In this scenario, the general formula derived in Theorem \ref{thm:first_intensity_stft} simplifies to:
    \begin{equation}\label{eq:rho_spec_purenoise}
    	\rhoTF(\tau, \omega) = 1 + \frac{(\log \smoothedpsd{S})''(\omega)}{4\pi}, \qquad (\tau, \omega) \in \bR^2.
    \end{equation}
    This expression decomposes the first intensity of zero-mean colored noise into two components: the constant $1$ associated with standard Complex White Noise, and a correction term strictly governed by the second logarithmic derivative of the smoothed PSD.
	Moreover, for certain classes of PSDs with slow decay, this correction term may become negligible as $\omega$ increases, justifying Flandrin's comment in \cite[Section 15.6]{flandrin2018explorations}, where it is stated that the zeros of the spectrogram ``remain almost unchanged'' compared to the white noise case. This is precisely what Theorem \ref{thm:proplimit} formalizes. 
    
    Before proving Theorem \ref{thm:proplimit}, we introduce an intermediate step. The main technical difficulty in analyzing \eqref{eq:rho_spec_purenoise} is the Gaussian convolution applied to the PSD. We show below that if the original PSD satisfies a sub-Gaussian lower bound and its derivatives decay faster than the PSD, then this property extends to the PSD convolved with a Gaussian. Furthermore, we quantify this decay.

    \begin{thm}\label{thm:logder_decay}
		Let $k \in \bN$. Let $\psd: \mathbb{R}\to[0,\infty)$ be $C^{k}(\bR) \cap L^{\infty}(\bR)$. Suppose that there exist $0<\gamma<2$, $\alpha,\beta > 0$, such that $\psd$ satisfies the following sub-Gaussian lower bound
		\begin{equation}\label{eq_xxx}
			\psd(\frequency) \geq \alpha e^{-\beta|\frequency|^\gamma},\qquad \forall\, \frequency \in \bR.
		\end{equation}
		Suppose that the $k$-th derivative of $\psd$ is asymptotically smaller than $\psd$ in the following sense:
		\begin{equation}\label{eq:S_deriv_sum}
			|\psd^{(k)}(\frequency)| \leq C(1+|\frequency|)^{-s} |\psd(\frequency)|, \qquad \forall\, \frequency \in \bR,
		\end{equation} for some $s>0$.
		Write $\smoothedpsd{S}(\frequency) := \psd*\phi\, (\frequency)$
		where $\phi(\frequency)=e^{-2\pi \frequency^2}$.
		
		Let $0<\varepsilon < (1-\gamma/2) s$.
		Then there exists a constant $K=K(C,s,\varepsilon,\alpha,\beta,\gamma)$ such that
		\begin{align}\label{eq_x}
			|\smoothedpsd{S}^{(k)}(\frequency)| \leq K\cdot \|\psd\|_\infty \cdot(1+|\frequency|)^{-\varepsilon} \smoothedpsd{S}(\frequency).
		\end{align}
	\end{thm}
	
	\begin{proof}
		{\bf Step 1}. Throughout this proof, all constants are allowed to depend on the parameters $C,s,\varepsilon,\alpha,\beta,\gamma$; this applies also to the constants implied in the $\lesssim$ notation. We claim that
		\begin{align}\label{eq_a}
			\smoothedpsd{S}(\frequency) \geq \alpha' e^{-\beta'|\frequency|^\gamma}
		\end{align}
		for constants $\alpha',\beta'>0$.
		Indeed, for $|\frequency|
		\geq 1$, we have $|\frequency|+1/2\leq 2|\frequency|$ and
		\begin{align}
			\smoothedpsd{S}(\frequency) &\geq \int_{-1/2}^{1/2} \psd(\frequency-\shift) \phi(\shift) \, d\shift
			\\
			&\gtrsim \int_{-1/2}^{1/2} e^{-\beta|\frequency-\shift|^\gamma} \, d\shift
			\geq \int_{-1/2}^{1/2} e^{-\beta(|\frequency|+1/2)^\gamma} \, d\shift
			\\
			&\geq e^{-\beta 2^{\gamma} |\frequency|^\gamma}.
		\end{align}
		On the other hand, for $|\frequency| \leq 1$
        and $|\tau| \leq 1$, we have $|\omega-\tau|\leq 2$ and
        \[
\smoothedpsd{S}(\omega)
\geq \alpha\int_{-1}^{1}
e^{-\beta|\omega-\tau|^\gamma}\phi(\tau)\,d\tau \geq
\alpha\int_{-1}^{1}
e^{-\beta2^\gamma}\phi(\tau)\,d\tau \gtrsim 1.
\]
Hence, $\smoothedpsd{S}(\frequency) \gtrsim 1$ for $|\frequency| \leq 1$, while the right-hand side of \eqref{eq_a} is $\lesssim 1$ for $|\frequency|
		\leq 1$.
		
		{\bf Step 2}. Let $\gamma < \gamma' < 2$ be such that $\varepsilon = (1-\gamma'/2)s$
        and note that
		\begin{align*}
			e^{-\pi |\frequency|^{\gamma'}} \lesssim e^{-\beta' |\frequency|^{\gamma}},
		\end{align*}
		so we conclude that there exists $\alpha''>0$ such that
		\begin{align}\label{eq_aa}
			\smoothedpsd{S}(\frequency) \geq \alpha'' e^{-\pi|\frequency|^{\gamma'}}.
		\end{align}
		
		{\bf Step 3}.
		Define
        \begin{equation*}
		\smoothedpsd{S}^*(\frequency) := \int_{\mathbb{R}} \psd(\frequency-\shift) (1+|\shift|)^{s} \phi(\shift)\,d\shift.
        \end{equation*}
		Note that $\smoothedpsd{S}^{(k)}=\psd^{(k)}*\phi$ and
		\begin{equation*}
		    (1+|\frequency-\shift|)^{-s} \leq (1+|\frequency|)^{-s} (1+|\shift|)^{s}.
		\end{equation*}
		Then
		\begin{align}
			|\smoothedpsd{S}^{(k)}(\frequency)| \leq \int_{\mathbb{R}} |\psd^{(k)}(\frequency-\shift)|
			\phi(\shift)\,d\shift
			&\leq C\int_{\mathbb{R}} (1+|\frequency-\shift|)^{-s} \psd(\frequency-\shift) \phi(\shift)\,d\shift
			\\
			&\leq C(1+|\frequency|)^{-s} \smoothedpsd{S}^*(\frequency).
		\end{align}
		On the other hand, for $M \geq 0$,
		\begin{align*}
			\smoothedpsd{S}^*(\frequency) &= \int_{|\shift|<M }\psd(\frequency-\shift) (1+|\shift|)^{s} \phi(\shift)\,d\shift
			+\int_{|\shift|>M }\psd(\frequency-\shift) (1+|\shift|)^{s} \phi(\shift)\,d\shift
			\\
			&\leq (1+M)^{s} \smoothedpsd{S}(\frequency)
			+\|\psd\|_\infty \int_{|\shift|>M } (1+|\shift|)^{s} \phi(\shift)\,d\shift
			\\
			&\leq C(1+M)^s \smoothedpsd{S}(\frequency) +C_{s} \|\psd\|_\infty e^{-\pi M^2}.
		\end{align*}
        We choose $M= |\frequency|^{\gamma'/2}$ to conclude that
		\begin{align*}
			\smoothedpsd{S}^*(\frequency) &\leq C'
			(1+|\frequency|)^{s\gamma'/2}\smoothedpsd{S}(\frequency) + C_{s} \|\psd\|_\infty e^{-\pi|\frequency|^{\gamma'}}
			\\
			&\leq C''\big[(1+|\frequency|)^{s\gamma'/2} + \|\psd\|_\infty \big]\smoothedpsd{S}(\frequency).
		\end{align*}
		
		{\bf Step 4}. By \eqref{eq_xxx}, $\|\psd\|_\infty \gtrsim 1$. Thus,
		\begin{align*}
			|\smoothedpsd{S}^{(k)}(\frequency)|
			&\lesssim (1+|\frequency|)^{-s}\big[(1+|\frequency|)^{s\gamma'/2} + \|\psd\|_\infty \big]\smoothedpsd{S}(\frequency)
			\\
			&\lesssim \|\psd\|_\infty \cdot (1+|\frequency|)^{-\varepsilon} \cdot \smoothedpsd{S}(\frequency),
		\end{align*}
		since $\varepsilon = (1-\gamma'/2)s$.
	\end{proof}

	\begin{proof}[Proof of Theorem \ref{thm:proplimit}]
    The PSD $\psd$ is continuous and bounded and has no zeros because of the sub-Gaussian lower bound. Hence $S$ is admissible and we can invoke Theorem \ref{thm:first_intensity_stft} to get
		\begin{align*}
			\rhoTF (\tau,\omega)
			=
			1 + \frac{(\log \smoothedpsd{S}  )''(\omega) } {4\pi}
			= 1
            + \frac{1}{4\pi}\bigg( \frac{(\smoothedpsd{S})''(\omega) }{ \smoothedpsd{S}(\omega) }
            - \frac{(\smoothedpsd{S}'(\omega))^2 }{{\smoothedpsd{S}(\omega)}^2}\bigg).
		\end{align*}
		  We apply then Theorem \ref{thm:logder_decay} by choosing $\theta \in (0,1)$ and setting $\varepsilon_k = \theta (1 - \frac{\gamma}{2} ) s_k$ for $k=1,2$. Since $\theta < 1$, these values satisfy the condition of Theorem \ref{thm:logder_decay}. Thus, \[|\smoothedpsd{S}^{(k)}(\omega)| \leq K \cdot\|\psd\|_\infty \cdot(1+|\omega|)^{-\varepsilon_k} \smoothedpsd{S}(\omega), \qquad k=1,2,\] for some constant $K=K(\theta, C_1, C_2, s_1, s_2, \alpha,\beta,\gamma)$.
          Since $\psd(0) \geq \alpha$, we can assume without loss of generality that $K \cdot\|\psd\|_\infty \geq 1$ and estimate
        \begin{align*}
        \biggabs{\frac{ \smoothedpsd{S}''(\omega) }{\smoothedpsd{S}(\omega)} - \frac{  {(\smoothedpsd{S}'(\omega))}^2 }{ {\smoothedpsd{S}(\omega)}^2  }}
        &\leq
        \biggabs{\frac{ \smoothedpsd{S}''(\omega) }{\smoothedpsd{S}(\omega)}}
        +\biggabs{\frac{  {\smoothedpsd{S}'(\omega)} }{ {\smoothedpsd{S}(\omega)}  }}^2 \leq
        2 K^2 \cdot \|\psd\|_\infty^2 \cdot (1+|\omega|)^{- \theta (1 - \frac{\gamma}{2}) \min(s_2,\, 2 s_1)},
        \end{align*}
		which proves the result.
	\end{proof}
	We now apply the previous result to rational power spectral densities.
	
	\begin{proof}[Proof of Proposition \ref{prop:decay_polinomial}]
    Let $\gamma\in(0,2)$. We apply Theorem \ref{thm:proplimit}. It is clear that $\psd \in C^2$ and, since $\alpha > 0$ and $\deg P < \deg Q$, $\psd$ is bounded.
    
    We will show the required sub-Gaussian lower bound $(P(\frequency)/Q(\frequency))^\alpha \ge a \exp(-b |\frequency|^\gamma)$ with $b=1$. The function $(P(\frequency)/Q(\frequency))^\alpha \, \exp(|\frequency|^\gamma)$ is continuous and strictly positive on $\bR$
    and diverges to $\infty$ as $|\frequency| \to \pm\infty$, so it has a global minimum $a > 0$.
    
    Let $T_1(\frequency) = \psd'(\frequency)/\psd(\frequency)$ and $T_2(\frequency) = \psd''(\frequency)/\psd(\frequency)$. For large $|\frequency|$, both $P'/P$ and $Q'/Q$ are $\mathcal{O}(|\frequency|^{-1})$. Therefore, $T_1(\frequency) = \alpha(P'/P - Q'/Q) = \mathcal{O}(|\frequency|^{-1})$, meaning that $(1+|\frequency|)T_1(\frequency)$ is bounded and \eqref{eq_intro_slow} holds with $s_1=1$. Similarly, $T_2(\frequency) = T_1'(\frequency) + (T_1(\frequency))^2 = \mathcal{O}(|\frequency|^{-2})$, and $(1+|\frequency|)^2 T_2(\frequency)$ is bounded, so \eqref{eq_intro_slow} holds with $s_2=2$.
    
    With $s_1=1$, and $s_2=2$, the decay exponent from Theorem \ref{thm:proplimit} is $\theta (2-\gamma)$. For any target decay rate $\eta \in (0,2)$, we can choose $\gamma = \frac{2-\eta}{2}\in(0,1)$ and $ \theta = \frac{2\eta}{2+\eta}\in(0,1) $, satisfying $\theta (2-\gamma) = \eta$. Theorem \ref{thm:proplimit} then yields:
    \begin{equation*}
        \abs{\rhoTF(\tau, \omega) - 1} \leq K \cdot\|\psd\|_\infty^2 \cdot(1+|\omega|)^{-\eta},
    \end{equation*}
        where $K=K(\theta, \gamma, \psd)$ which proves the result.
    \end{proof}
	
    Finally, let us prove Proposition \ref{prop:gaussianpsd}, concerning the sharpness of Theorem \ref{thm:proplimit}.
    
   \begin{proof}[Proof of Proposition \ref{prop:gaussianpsd}]
   We compute explicitly the smoothed PSD, $\smoothedpsd{S}(\omega) = (\psd * \phi)(\omega)$. Since $\phi(t) = e^{-2\pi t^2}$ and $\psd(t) = e^{-b t^2}$, we get a properly normalized Gaussian whose mean and variance are the sum of the convolved ones \cite[Section 4]{bromiley2003products}:
    \begin{equation*}
        \smoothedpsd{S}(\omega) = \sqrt{\frac{\pi}{2\pi + b}} \exp\biggl( - \frac{2\pi b}{2\pi + b} \omega^2 \biggr).
    \end{equation*}   
    The first intensity reads then:
    \begin{equation*}
        \rhoTF(\tau, \omega) = 1 + \frac{1}{4\pi} (\log(\smoothedpsd{S}))''(\omega) = 1 + \frac{1}{4\pi} \biggl( - \frac{4\pi b}{2\pi + b} \biggr) = 1 - \frac{b}{2\pi + b} = \frac{2\pi}{2\pi + b}.
    \end{equation*}
    \end{proof}

	\section{Numerical Framework and PSD Recovery} \label{sec:simulations}

    The objective of this section is twofold. On the one hand, we propose and implement an algorithm to estimate the smoothed PSD $\smoothedpsd{S}$ based solely on the statistics of the spectrogram zeros. On the other hand, we show how we simulate the point process and validate the formulas derived in the previous sections. Specifically, we describe: how $\smoothedpsd{S}$ can be recovered from the zero set of the STFT, how we define the discrete STFT used in our experiments, and what experimental setup is used to corroborate our results.
    
	\subsection{Estimating the smoothed PSD from the zeros of the spectrogram}\label{sec:zerostopsd}
	
	Our goal is to recover the smoothed PSD $\smoothedpsd{S}$ from the statistics of the zero set
    of the spectrogram of colored noise $\cn$. From \eqref{eq:rho_spec_purenoise}, $\rhoTF(\tau, \omega) - 1 = \frac{(\log \smoothedpsd{S}(\omega))''}{4\pi}$; integrating twice gives us access to $\smoothedpsd{S}$.
    
    \begin{remark}[Identifiability of the smoothed PSD from the zero set]\label{rem:identifiability}
    	Since $\rhoTF$ depends on $\smoothedpsd{S}$ only through
    	$(\log \smoothedpsd{S})''$, two smoothed PSDs differing by a
    	factor $e^{C_1\omega + C_0}$ are indistinguishable from the zero
    	set. Recovering $\smoothedpsd{S}$ by double integration therefore
    	requires fixing two integration constants, which cannot be determined
    	from the zeros alone. To resolve this ambiguity, we must fix the value and slope of $\log \smoothedpsd{S}$ at $\omega=0$.
    \end{remark}
    
	Recall that the integral of $\rhoTF(\tau, \omega)$ over any Borel region of the time-frequency plane equals the expected number of spectrogram zeros therein.
	More formally, for Borel $\Omega \subseteq \bR^2$:
    \begin{equation}\label{eq:remark1_rho}
        \int_{\Omega} \rhoTF(\tau, \omega)\, d\tau d\omega = \E{N_{\text{Spec}_{\cn}}(\Omega)}.
    \end{equation}
	Since in our case $\rhoTF$ depends only on the frequency coordinate $\omega$, we have that for $L>0$ and $\il \in \bR$, writing $I_{\il} \coloneqq [0,\il]$ if $\il \geq 0$ and $I_{\il} \coloneqq [\il,0]$ if $\il<0$:
	\begin{align}\label{eq:remark2_rho}
		\int_{[-L,L] \times I_{\il}} \rhoTF(\tau, \omega)\, d\tau d\omega = 2L \, \sgn(\il) \int_{0}^{\il} \rhoTF(\tau, \omega)\, d\omega.
	\end{align}
	
    To recover $\log(\smoothedpsd{S})$ from $\rhoTF$, we impose the following normalization conditions at $\omega=0$:
	\begin{align}\label{eq:conditions_smooth_normalization}
		\log\smoothedpsd{S}(0) = 0, \qquad (\log\smoothedpsd{S})'(0) = 0.
	\end{align}
	Then, for $\ill \in \bR$, integrating the intensity difference twice yields:
	\begin{equation}
		\int_{0}^{\ill} \int_{0}^{\il} (\rhoTF(\tau, \omega) - 1)\, d\omega\, d\il
		= \frac{1}{4\pi}   \int_{0}^{\ill}  (\log(\smoothedpsd{S}))'(\il)\, d\il
		= \frac{1}{4\pi} \log(\smoothedpsd{S}(\ill)).\label{eq:int_lhs}
	\end{equation}
	Equivalently, using \eqref{eq:remark1_rho} and \eqref{eq:remark2_rho}, the same double integral can be written in terms of the expected number of zeros:
    \begin{align}
        \int_{0}^{\ill} \int_{0}^{\il} \left( \rhoTF(\tau, \omega) - 1 \right) \, d\omega\, d\il
        &=
         \int_{0}^{\ill} \frac{\sgn(\ill)}{2L}\, \E{N_{\text{Spec}_{\cn}}([-L, L]\times I_{\il})} \, d\il - \frac{{\ill}^2}{2}. \label{eq:int_rhs2}
    \end{align}
    Putting together \eqref{eq:int_lhs} and \eqref{eq:int_rhs2} and isolating $\smoothedpsd{S}(\ill)$:
    \begin{equation}
        \smoothedpsd{S}(\ill) = \exp\biggl(\int_{0}^{\ill} \frac{2\pi}{L} \, \sgn(\ill) \, \E{N_{\text{Spec}_{\cn}}([-L, L]\times I_{\il})} \, d\il - 2\pi {\ill}^2 \biggr).
        \label{eq:cnot}
    \end{equation}
    Equation~\eqref{eq:cnot} forms the basis of our estimator for $\smoothedpsd{S}(\ill)$, obtained by replacing the expected zero count with an empirical count.
        
    \subsection{Simulation of colored noise}\label{section:simcolnoise}

    For the numerical implementation, we simulate signals on the interval $[-L, L]\subseteq \bR$, $L\in\bN$, using $K = 4L^2$ samples. The mapping from discrete to continuous coordinates is $\mapcoord(j) \coloneqq -L + j\gridstep$ for $0 \le j < K$, with $\gridstep = 1/(2L)$. To color the noise, we use the discrete Fourier transform (DFT) and its inverse (IDFT), both normalized unitarily: $\fft{\xvect}_{k} = \frac{1}{\sqrt{K}}\sum_{j=0}^{K-1} \xel_{j} e^{-2\pi i \frac{jk}{K}}$ and $\ifft{\xvect}_{k} = \frac{1}{\sqrt{K}}\sum_{j=0}^{K-1} \xel_{j} e^{2\pi i \frac{jk}{K}}$, for $0 \leq k \leq K-1$.
    
    The discrete frequency vector is constructed by circularly shifting the frequencies mapped by $\mapcoord$ so that $k=0$ corresponds to zero frequency, following standard DFT convention: $\frequency_k \coloneqq \mapcoord\bigl((k + 2L^2) \pmod{4L^2}\bigr)$, for $0 \le k \le 4L^2 - 1$. Note that $K = 4L^2$ ensures symmetric resolution between the time and frequency domains: both the temporal sampling resolution and the DFT frequency resolution equal $\gridstep = 1/(2L)$.

    We generate discrete colored noise $\simnoise$ by taking the IDFT of complex white Gaussian noise element-wise multiplied by the square root of the PSD $\psd$. Namely, the discrete complex white noise is constructed from independent real standard normal vectors $\boldsymbol{\xi}_1, \boldsymbol{\xi}_2 \sim \mathcal{N}(0, 1)$ as $\simwhite = \frac{1}{\sqrt{2}}(\boldsymbol{\xi}_1 + i \boldsymbol{\xi}_2)$.
    To properly match the variance of continuous colored noise, we scale it by $1/\sqrt{\gridstep}$, that is, $\sqrt{2L}$:
    \begin{equation}
        \simnoise \coloneqq \sqrt{2L} \, \ifft{\simwhite \odot \sqrt{\psd(\boldsymbol{\frequency})}}.
    \end{equation}
    The dependence on $S$ is suppressed in the notation.

    \noindent\textbf{Covariance of the generated noise:} Expanding the definition of the IDFT, the sample of the noise at index $s$ is:
    \begin{equation}
        \simnoise_s = \frac{1}{\sqrt{2L}} \sum_{r=0}^{4L^2-1} \simwhite_r \sqrt{\psd(\frequency_r)} e^{2\pi i \frac{rs}{4L^2}}.
    \end{equation}
    We compute the covariance between two temporal samples $\simnoise_s$ and $\simnoise_p$. Since $\simwhite$ is independent complex white noise, $\E{\simwhite_r \overline{\simwhite_l}} = \delta_{r,l}$, which yields the following computation:
    \begin{align}
    \E{\simnoise_s \overline{\simnoise_p}} 
                                            &= \biggE{ \biggl( \frac{1}{\sqrt{2L}} \sum_{r=0}^{4L^2-1} \simwhite_r \sqrt{\psd(\frequency_r)} e^{2\pi i \frac{rs}{4L^2}} \biggr) \overline{\biggl( \frac{1}{\sqrt{2L}} \sum_{l=0}^{4L^2-1} \simwhite_l \sqrt{\psd(\frequency_l)} e^{2\pi i \frac{lp}{4L^2}} \biggr)} }, \nonumber \\
                                            &=\frac{1}{2L} \sum_{r=0}^{4L^2-1} \psd(\frequency_r) e^{2\pi i \frac{r(s-p)}{4L^2}}.
    \end{align}
    
    Due to the periodic wrapping of the frequencies, the difference $\frac{r}{4L^2}(s-p) - \frequency_r (\mapcoord(s) - \mapcoord(p))$ is an integer. In fact, from the definition of $\frequency_r$, we know that for each index $r \in \{0, \dots, K-1\}$, there exists an integer $m_r$ such that:
    \begin{equation*}
        \frequency_r = -L + \frac{r + 2L^2 - m_r (4L^2)}{2L} = -L + \frac{r}{2L} + L - 2m_r L = \frac{r}{2L} - 2m_r L.
    \end{equation*}
    Using the fact that $\mapcoord(s) - \mapcoord(p) = (s-p)\gridstep = \frac{s-p}{2L}$, we evaluate the product:
    \begin{equation*}
        \frequency_r (\mapcoord(s) - \mapcoord(p)) = \left( \frac{r}{2L} - 2m_r L \right) \frac{s-p}{2L} = \frac{r(s-p)}{4L^2} - m_r(s-p),
    \end{equation*}
    where $m_r (s-p) \in \mathbb{Z}$. Consequently, $e^{2\pi i \frequency_r (\mapcoord(s) - \mapcoord(p))} = e^{2\pi i \frac{r(s-p)}{4L^2}}$ and therefore:
    \begin{equation}\label{eq:discrete_noise_cov}
        \E{\simnoise_s \overline{\simnoise_p}} =  \frac{1}{2L} \sum_{r=0}^{4L^2-1} \psd(\frequency_r) e^{2\pi i \frequency_r (\mapcoord(s) - \mapcoord(p))}.
    \end{equation}
    For a continuous and integrable PSD whose Fourier transform decays adequately, \eqref{eq:discrete_noise_cov} is a Riemann sum for
    $\fourier{\psd}(\mapcoord(p)-\mapcoord(s))$ with the sampling resolution $\frac{1}{2L}$.
    
    \subsection{Discrete STFT}\label{sec:discretestft}
    To simulate the continuous STFT in $\bR^2$, we restrict the plane to the bounded square $[-L, L]^2$ and introduce a grid $\Lambda_{L}$ defined as:
    \begin{equation}\label{eq:lattice}
        \Lambda_L \coloneqq \Bigl\{ (\tau_k, \omega_j) = \bigl(\mapcoord(k), \mapcoord(j)\bigr) : 0 \leq k,j \leq 4L^2 - 1 \Bigr\} \subset [-L,L]^2.
    \end{equation}
    The distance between adjacent nodes in this grid is exactly the step size $\gridstep$ in both the time and frequency directions. Given a discrete signal vector $\vec{f} \in \bC^{4L^2}$, its discrete STFT with a normalized Gaussian window $\window(t) = 2^{1/4} e^{-\pi t^2}$ is evaluated at each node $\lambda = (\tau_k, \omega_j) \in \Lambda_L$ as:
    \begin{equation}\label{eq:discretizedstft}
        \discreteSTFT{\vec{f}}(\lambda) \coloneqq \gridstep \sum_{s=0}^{4L^2-1} 2^{1/4} \, \vec{f}_s \, e^{-\pi(\tau_s - \tau_k)^2} e^{-2\pi i \tau_s \omega_j}, \quad \lambda \in \Lambda_L,\, \tau_s = \mapcoord(s).
    \end{equation}    
    If $\vec{f}_s = f(\tau_s)$ are the samples of a continuous integrable function $f$, then \eqref{eq:discretizedstft} is a Riemann sum approximation of the continuous transform $\stft{\window}{f}(\tau_k, \omega_j)$.

Near the boundaries the discretization error corresponding to \eqref{eq:discretizedstft} may be significant due to the truncation of the Gaussian window to $[-L, L]$. To prevent this, the simulations incorporate several precautions detailed below.
    
	\subsection{Consistency of the Simulation}

    We now show that the statistical properties of the discrete STFT converge to those of the continuous model as $L \to \infty$. Since a physical point in the time-frequency plane need not lie exactly on the discrete lattice, we evaluate the transform at the nearest available grid point.

    \begin{lemma}
        Let $L\in \bN$. Let $\psd$ be an admissible and continuous PSD. Let $\simnoise \in \bC^{4L^2}$ be a complex circularly symmetric Gaussian vector whose covariance is given by
        \begin{equation}\label{eq:covariance_lemma}
            \E{\simnoise_s \overline{\simnoise_p}} =  \frac{1}{2L} \sum_{r=0}^{4L^2-1} \psd(\frequency_r) e^{2\pi i \frequency_r (\mapcoord(s) - \mapcoord(p))}.
        \end{equation}
        Let $\lambda_A = (\tau_A, \omega_A)$ and $\lambda_B = (\tau_B, \omega_B)$ be two arbitrary fixed coordinates in the time-frequency plane $\bR^2$. Let $\lambda_A^{(L)}$ and $\lambda_B^{(L)}$ denote (a choice of) respective nearest points on the discrete spatial grid $\Lambda_L$ associated with the temporal resolution $\gridstep = \frac{1}{2L}$. As $L \to \infty$, the covariance of the discrete STFT converges to the theoretical continuous covariance of the STFT:
        \begin{equation} \label{eq:covariancestftrescaled}
            \lim_{L \to \infty} \E{\discreteSTFT{\simnoise}(\lambda_A^{(L)}) \, \overline{\discreteSTFT{\simnoise}(\lambda_B^{(L)})}} = \E{\stft{\window}{\cn}(\tau_A, \omega_A) \, \overline{\stft{\window}{\cn}(\tau_B, \omega_B)}}.
        \end{equation}
    \end{lemma}
    
    \begin{proof}
        Based on \eqref{eq:discretizedstft}, the discrete STFT at $\lambda = (\tau, \omega) \in \bR^2$ can be rewritten as the complex inner product against $\Psi_{\lambda}(\mapcoord(\cdot))$, where $\Psi_{\lambda} \coloneqq \mo{\omega} \tr{\tau} \window$. Evaluating the discrete STFT at the noise vector, expanding the covariance of the transforms and substituting \eqref{eq:covariance_lemma} yields:
        \begin{equation}\label{eq:expected_discrete_cov}
            \E{\discreteSTFT{\simnoise}(\lambda_A^{(L)}) \, \overline{\discreteSTFT{\simnoise}(\lambda_B^{(L)})}} = \gridstep^2 \sum_{s=0}^{4L^2-1} \sum_{p=0}^{4L^2-1} \overline{\Psi_{\lambda_A^{(L)}}(\mapcoord(s))} \Psi_{\lambda_B^{(L)}}(\mapcoord(p)) \left( \gridstep \sum_{r=0}^{4L^2-1} \psd(\frequency_r) e^{2\pi i \frequency_r (\mapcoord(s) - \mapcoord(p))} \right).
        \end{equation}
        Thus, the sum in the above equation reads
        \begin{equation}\label{eq:covariance_disc_stft}
            \sum_{r=0}^{4L^2-1} \gridstep \, \psd(\frequency_r) \left( \gridstep \sum_{s=0}^{4L^2-1} \overline{\Psi_{\lambda_A^{(L)}}(\mapcoord(s))} e^{2\pi i \frequency_r \mapcoord(s)} \right) \overline{\left( \gridstep \sum_{p=0}^{4L^2-1} \overline{\Psi_{\lambda_B^{(L)}}(\mapcoord(p))} e^{2\pi i \frequency_r \mapcoord(p)} \right)}.
        \end{equation}
        For $\xi \in \bR$, we set
        \begin{equation}\label{eq:defR}
            R_{A,L}(\xi) \coloneqq \gridstep \sum_{s=0}^{4L^2-1} \overline{\Psi_{\lambda_A^{(L)}}(\mapcoord(s))} e^{2\pi i \xi \mapcoord(s)},
        \end{equation}
        and analogously $R_{B,L}$, so that \eqref{eq:covariance_disc_stft} reads $\sum_{r} \gridstep\, \psd(\frequency_r)\, R_{A,L}(\frequency_r)\, \overline{R_{B,L}(\frequency_r)}$.

        As $L \to \infty$ the grid resolution $\gridstep \to 0$ and $\abs{\lambda_A - \lambda_A^{(L)}} \to 0$, so one expects $R_{A,L}(\frequency_r) \to \overline{\fourier{\Psi_{\lambda_A}}(\frequency)}$ and $\overline{R_{B,L}(\frequency_r)} \to \fourier{\Psi_{\lambda_B}}(\frequency)$, and the outer sum in \eqref{eq:covariance_disc_stft} to converge to the corresponding integral. Making this precise amounts to controlling two errors, which we now isolate.

        \medskip
        \noindent\emph{Step 1: decomposition of $R_{A,L}$ and $R_{B,L}$.}
        Write $\Psi \coloneqq \Psi_{\lambda_A^{(L)}}$. Since $\mapcoord(s) = (s - 2L^2)\gridstep$, the sampling points are exactly the points of the lattice $\gridstep\bZ$ lying in the window $[-L,L)$, so \eqref{eq:defR} is a truncated lattice sum. Completing it to the full lattice and applying the Poisson summation formula to the Schwartz function $u \mapsto \overline{\Psi(u)}\, e^{2\pi i \xi u}$, gives
        \begin{equation}\label{eq:poisson}
            \gridstep \sum_{s\in\bZ} \overline{\Psi(s\gridstep)}\, e^{2\pi i \xi s\gridstep}
            = \sum_{m\in\bZ} \overline{\fourier{\Psi}(\xi - 2Lm)},
            \qquad \xi \in \bR,
        \end{equation}
        and therefore
        \begin{equation}\label{eq:decomposition}
            R_{A,L}(\xi) \;=\;
            \overline{\fourier{\Psi}(\xi)}
            \;+\;
            \sum_{m \in \mathbb{Z}\setminus\{0\}} \overline{\fourier{\Psi}(\xi - 2Lm)}
            \;-\;
            \gridstep \!\!\!\sum_{\substack{t \in \gridstep\bZ \\ t \geq L \text{ or } t \le -L-\gridstep}}\!\!\! \overline{\Psi(t)}\, e^{2\pi i \xi t},
        \end{equation}
        and likewise for $\lambda_B$. 
        
        Our goal is to approximate $R_{A,L}(\xi)$ by $\overline{\fourier{\Psi}(\xi)}$, when $\frequency_r \in [-L,L)$. By \eqref{eq:decomposition}, it suffices to show that if $\frequency_r \in [-L,L)$, the extra-terms are exponentially small in $L$.

        \medskip
        \noindent\emph{Step 2: bounds.}
        Here $X \lesssim Y$ means $X \le C\,Y$ for a constant $C > 0$ that depends only on the number $M$ fixed below. In particular, $C$ never depends on $L$, on $\xi$, or on the summation index $m$.
        
        Since $\lambda_A^{(L)} \to \lambda_A$ and $\lambda_B^{(L)} \to \lambda_B$, we can fix $M > 0$ such that $\abs{\tau_A^{(L)}}$, $\abs{\omega_A^{(L)}}$, $\abs{\tau_B^{(L)}}$ and $\abs{\omega_B^{(L)}}$ are $\le M$ for all $L$, and assume from now on that $L \geq 2M+2$. Recall that $\bigabs{\fourier{\Psi_{\lambda}}(\xi)} = 2^{1/4} e^{-\pi(\xi-\omega)^2}$, and let $\abs{\xi} \le L$.

        \emph{(I).} Since $\bigl(\xi-\omega_A^{(L)}\bigr)^2 \geq \tfrac{1}{2}\xi^2 - M^2$, we have
        \begin{equation}\label{eq:bound_I}
            \bigabs{\fourier{\Psi}(\xi)} \;\lesssim\; e^{-\pi \xi^2/2} .
        \end{equation}

        \emph{(II).} For $m \in \mathbb{Z}\setminus\{0\}$, the $m$-th alias is centered far from the evaluation window, namely:
        \begin{equation*}
            \bigabs{\xi - 2Lm - \omega_A^{(L)}} \geq 2L\abs{m} - \abs{\xi} - M \geq L\bigl(2\abs{m} - \tfrac{3}{2}\bigr) \geq \tfrac{1}{2} L \abs{m},
        \end{equation*}
        and, since $\abs{\xi} \le L$, the same quantity is $\geq \tfrac{1}{2}\abs{\xi}\abs{m}$. Thus $\bigl(\xi - 2Lm - \omega_A^{(L)}\bigr)^2 \geq m^2(L^2+\xi^2)/8$. Thus,
        \begin{equation}\label{eq:psf_aliases}
            \sum_{m \in \mathbb{Z}\setminus\{0\}} \bigabs{\fourier{\Psi}(\xi - 2Lm)} \;\lesssim\; e^{-\pi L^2/8}\, e^{-\pi \xi^2/8} .
        \end{equation}

        \emph{(III).} In order to bound the truncation error in \eqref{eq:decomposition}, we note that $\bigabs{\overline{\Psi(t)}e^{2\pi i \xi t}} = \window\bigl(t - \tau_A^{(L)}\bigr)$, and that, if  $\abs{t} \geq L$ then $\abs{t - \tau_A^{(L)}} \geq L - M \geq \tfrac{1}{2}L$. Comparing the sum with the corresponding Gaussian integral,
        \begin{equation}\label{eq:truncation}
             \gridstep \!\!\!\sum_{\substack{t \in \gridstep\bZ \\ t \geq L \text{ or } t \le -L-\gridstep}}\!\!\! \bigabs{\Psi(t)}
            \;=\;
            2^{1/4}\, \gridstep \!\!\!\sum_{\substack{t \in \gridstep\bZ \\ t \geq L \text{ or } t \le -L-\gridstep}}\!\!\! e^{-\pi (t-\tau_A^{(L)})^2}
            \;\lesssim\;
            \int_{\abs{u} \geq L/2} e^{-\pi u^2}\, du
            \;\lesssim\; e^{-\pi L^2/8}.
        \end{equation}

        Combining \eqref{eq:bound_I}--\eqref{eq:truncation} in \eqref{eq:decomposition}, and using $e^{-\pi L^2/8} \le e^{-\pi \xi^2/8}$ for $\abs{\xi} \le L$, we obtain
        \begin{equation}\label{eq:envelope}
            \abs{R_{A,L}(\xi)} \;\lesssim\; e^{-\pi \xi^2/8},
            \qquad L \geq 2M+2, \quad \abs{\xi} \le L,
        \end{equation}
        and the same bound applies to $\abs{R_{B,L}(\xi)}$.
        
        \medskip
        \noindent\emph{Step 3: dominated convergence.}
        We set $I_r^{(L)} \coloneqq [\frequency_r, \frequency_r + \gridstep)$ and construct the step functions
        \begin{equation*}
            F_L(\frequency) \coloneqq \sum_{r=0}^{4L^2-1} \psd(\frequency_r) \, R_{A,L}(\frequency_r) \, \overline{R_{B,L}(\frequency_r)}\,
            \mathbf{1}_{I_r^{(L)}}(\frequency),
        \end{equation*}
        from which one verifies $\int_\bR F_L(\frequency)\, d\frequency = \E{\discreteSTFT{\simnoise}(\lambda_A^{(L)}) \, \overline{\discreteSTFT{\simnoise}(\lambda_B^{(L)})}}$.

        To show pointwise convergence, we fix $\frequency \in \bR$ and let $L > \max\{\abs{\frequency}, 2M+2\}$. Let $r = r(L,\frequency)$ be the index with $\frequency \in I_r^{(L)}$, so $\abs{\frequency_r - \frequency} \le \gridstep \to 0$. By \eqref{eq:psf_aliases} and \eqref{eq:truncation}, the decomposition \eqref{eq:decomposition} reads
        \begin{equation*}
            R_{A,L}(\frequency_r) = \overline{\fourier{\Psi_{\lambda_A^{(L)}}}(\frequency_r)} + O\bigl(e^{-\pi L^2/8}\bigr),
        \end{equation*}
        while $(\lambda,\eta) \mapsto \fourier{\Psi_{\lambda}}(\eta)$ is continuous. Since $\lambda_A^{(L)} \to \lambda_A$ and $\frequency_r \to \frequency$, we get $R_{A,L}(\frequency_r) \to \overline{\fourier{\Psi_{\lambda_A}}(\frequency)}$ and $\overline{R_{B,L}(\frequency_r)} \to \fourier{\Psi_{\lambda_B}}(\frequency)$; and $\psd(\frequency_r) \to \psd(\frequency)$ by continuity of $\psd$. So $F_L(\frequency) \to \psd(\frequency) \overline{\fourier{\Psi_{\lambda_A}}(\frequency)} \fourier{\Psi_{\lambda_B}}(\frequency)$ for every $\frequency \in \bR$.

        Now we exhibit an integrable function that bounds $F_L$. Since $\psd$ is admissible, there are $n>0$ and $C_{\psd}>0$ with $\psd(\frequency) \le C_{\psd}(1+\abs{\frequency})^{n}$ for all $\frequency \in \bR$. Let $L \geq 2M+2$, $\frequency \in [-L,L)$ and $r$ as above.
        
        Since $\abs{\frequency_r} \le L$, we apply \eqref{eq:envelope} obtaining
        \begin{equation*}
            \abs{F_L(\frequency)} = \psd(\frequency_r)\,\abs{R_{A,L}(\frequency_r)}\,\abs{R_{B,L}(\frequency_r)} \; \lesssim \; (1+\abs{\frequency_r})^n\, e^{-\pi \frequency_r^2/4}.
        \end{equation*}

        To replace $\frequency_r$ by $\frequency$ we note that $\abs{\frequency - \frequency_r} \le \gridstep \le 1$, $1 + \abs{\frequency_r} \le 2 + \abs{\frequency}$ and $\frequency_r^2 \geq \frequency^2/2 - 1$. Thus, $e^{-\pi \frequency_r^2/4} \le e^{\pi/4}\, e^{-\pi \frequency^2/8}$ and therefore
        \begin{equation}\label{eq:domination}
            \abs{F_L(\frequency)} \le \Phi(\frequency) \coloneqq C\,(2+\abs{\frequency})^{n}\, e^{-\pi \frequency^2/8},
            \qquad \frequency \in [-L,L), \quad L \geq 2M+2,
        \end{equation}
        where $C$ does not depend on $L$ or $\frequency$. Since $F_L$ vanishes off $[-L,L)$, the bound holds for any $\frequency\in\bR$. 

        Since $\Phi \in L^1(\bR)$, the limit passes inside the integral by dominated convergence:
        \begin{equation}\label{eq:discrete_limit_covariance}
            \lim_{L \to \infty} \E{\discreteSTFT{\simnoise}(\lambda_A^{(L)}) \, \overline{\discreteSTFT{\simnoise}(\lambda_B^{(L)})}} = \int_{\bR} \lim_{L \to \infty} F_L(\frequency) \, d\frequency = \int_{\bR} \psd(\frequency) \overline{\fourier{\Psi_{\lambda_A}}(\frequency)} \fourier{\Psi_{\lambda_B}}(\frequency) \, d\frequency,
        \end{equation}
        which is exactly \eqref{eq:covariance_noise} applied to $\stft{\window}{\cn}(\tau_A, \omega_A) = \ip{\cn}{\Psi_{\lambda_A}}$ and $\stft{\window}{\cn}(\tau_B, \omega_B) = \ip{\cn}{\Psi_{\lambda_B}}$.
    \end{proof}
	
    \subsection{Experimental setup and numerical estimators}
    In all our experiments, the STFT of the colored noise is simulated according to \eqref{eq:discretizedstft}. To mitigate boundary effects, we simulate the discrete STFT over an extended grid $[-L_{sim}, L_{sim}]^2$ with $L_{sim} = 50$ and resolution $\gridstep = 0.01$. However, we restrict our zero-counting analysis to the inner domain $[-L, L]^2$ with $L = 25$. This margin reduces truncation and aliasing artifacts and refines the approximation of the theoretical continuous covariance.
    
    In addition, we define a parametric model for the PSDs that are considered in our simulations:
    \begin{equation}\label{eq:psd_model}
        \psd_{p,c}(\frequency) = \frac{K_{p,c}}{\left(c+\frequency^2\right)^{p}},
    \end{equation}
    where $p,c>0$ and the constant $K_{p,c}>0$ is chosen so that
    $(\psd_{p,c})_\phi(0)=
    \int_{-\infty}^{\infty} e^{-2\pi \frequency^2} \psd_{p,c}(\frequency) \, d\frequency = 1$, so that the first condition in \eqref{eq:conditions_smooth_normalization} holds. The second condition in \eqref{eq:conditions_smooth_normalization} also holds because $\psd_{p,c}$ is even:
    $(\psd_{p,c})'_\phi(0)= 4\pi\int_{-\infty}^{\infty} \frequency e^{-2\pi \frequency^2} \psd_{p,c}(\frequency)\, d\frequency=0$.

    \subsubsection{Estimating the first intensity}\label{sec:exp1}
    For a given $\numsim \in \bN$, we simulate $\simnoise_1, \ldots, \simnoise_\numsim$ realizations of complex colored noise with PSD $\psd$ as in \eqref{eq:psd_model} and apply the discrete STFT to each of them as defined in \eqref{eq:discretizedstft}. We extract the zero set using the Minimal Grid Neighbors (MGN) method, introduced in \cite{flandrin2015time}. In this approach, the zeros of $\discreteSTFT{\simnoise_l}$ are identified as the grid points where the magnitude $\abs{\discreteSTFT{\simnoise_l}}$ attains a local minimum; see \cite{efficient} for analysis of the computation of the zero set of the Gabor transform of signals and noise. 
    
    When a grid-point $\lambda\in\Lambda_L$ is selected by the MGN algorithm, it is likely that a spectrogram zero is present in the box $Q_\lambda \coloneqq \lambda+[-\gridstep/2,\gridstep/2]^2$. To account for boundary effects of the count of zeros in an observed window $\Omega\subseteq \bR^2$, each detected zero shall be weighted by the fraction $|Q_\lambda\cap\Omega|/\gridstep^2$. For most bins, the endpoints are placed midway between grid nodes, and this yields ordinary counts. Formally, given an observation window $\Omega \subseteq [-L,L]^2$ we denote by $\mathcal{X}_l \subseteq \Lambda_L$ the set of grid points selected by MGN for $\discreteSTFT{\simnoise_l}$ and define
    \begin{equation*}
    N_{\discreteSTFT{\simnoise_l}}(\Omega) 
    \coloneqq
    \sum_{\lambda \in \mathcal{X}_l} \frac{\abs{Q_\lambda \cap \Omega}}{\gridstep^2}
    =
    \# \bigl\{\lambda \in \mathcal{X}_l : Q_\lambda \subseteq \Omega\bigr\}
+ \sum_{\substack{\lambda \in \mathcal{X}_l \\ Q_\lambda \not\subseteq \Omega}} \frac{\abs{Q_\lambda \cap \Omega}}{\gridstep^2},
    \end{equation*}
    where $\abs{\cdot}$ is the Lebesgue measure on $\bR^2$.
    
    Although MGN is mainly used in white noise scenarios ($\psd \equiv 1$), the main intuition behind the method is the analyticity of the Bargmann transform, which also holds for colored noise. The results of our comparison between numerical estimators and derived theoretical quantities (see below) empirically validate the use of MGN when applied to colored noise with a mildly varying $\psd$.
  
    For a partition $I_1, \ldots, I_m$ of the frequency axis $[-L, L]$, we estimate $\rhoTF(\omega)$ using a histogram normalized by the number of realizations and the size of the rectangular strips where zeros are counted. More precisely:
    \begin{equation}\label{eq:histogram}
        \hatrhoTF(\omega) =
        \sum_{l=1}^{\numsim}
        \sum_{\substack{j=1 \\ \omega \in I_j}}^{m}
        \frac{N_{\discreteSTFT{\simnoise_{l}}}([-L,L] \times I_j)}{\numsim \cdot 2L \cdot |I_j|},\qquad \frequency\in [-L,L],
    \end{equation}
    which we compare with the true spectrogram intensity. The quantity $N_{\discreteSTFT{\simnoise_{l}}}([-L,L] \times I_j)$ is computed with the MGN algorithm.
    
    As a benchmark of this estimator, the true first intensity \eqref{eq:rho_spec_purenoise} is computed by expanding $\log(\smoothedpsd{S})''$ and computing the necessary integrals numerically.
    
    \subsubsection{Estimation of the smoothed PSD out of zeros}\label{sec:exp2}
    We now estimate the smoothed PSD $\smoothedpsd{S}$ via the identity in \eqref{eq:cnot}: for every frequency $\omega \in \bR$,
    \begin{equation}\label{eq:exp2_coreformula}
        \smoothedpsd{S}(\omega) = (\phi * \psd) (\omega) = \exp\left(\sgn(\omega) \int_{0}^{\omega} \frac{2\pi}{L} \,  \, \E{N_{\text{Spec}_{\cn}}([-L, L] \times I_\nu)}\, d\nu - 2\pi {\omega}^2 \right),
    \end{equation}
    where $\phi(t) = e^{-2 \pi t^{2}}$ and $I_\nu$ is as in \eqref{eq:remark2_rho}.
    
   Let $\empintegral_l(\omega)$ denote the empirical integral of the zero-count for a single realization $l$:
    \begin{equation*}
        \empintegral_l(\omega) = \frac{2\pi}{L} \sgn(\omega) \int_{0}^{\omega}  \, N_{\discreteSTFT{\simnoise_{l}}}([-L, L] \times I_\nu) \, d\nu.
    \end{equation*}
    The average integral across all realizations is $\bar{\empintegral}(\omega) = \frac{1}{\numsim} \sum_{l=1}^{\numsim} \empintegral_l(\omega)$. Let $m_\omega=\mathbb E[\empintegral_1(\omega)]$ and
$v_\omega=\operatorname{Var}(\empintegral_1(\omega))$. Since
$\bar \empintegral(\omega)=\frac1{\numsim}\sum_{l=1}^{\numsim} \empintegral_l(\omega)$
is the average of i.i.d. random variables, the central limit theorem gives
$\bar \empintegral(\omega)\approx
\mathcal N\left(m_\omega,\frac{v_\omega}{\numsim}\right)$.
Consequently, $\exp(\bar \empintegral(\omega))$
is approximately log-normal, with mean approximately
\[
\exp\left(m_\omega+\frac{v_\omega}{2\numsim}\right).
\]
To improve the finite sample performance, we estimate $v_\omega$ by
\[
\hat{\sigma}_\empintegral^2(\omega)
=
\frac{1}{\numsim-1}
\sum_{l=1}^{\numsim}
\left(\empintegral_l(\omega)-\bar \empintegral(\omega)\right)^2,
\]
and incorporate the log-normal bias correction by defining
\begin{equation}\label{eq_new_est}
\widehat{\smoothedpsd{S}}(\omega)
=
\exp\Big(
\bar \empintegral(\omega)-2\pi\omega^2
-\frac{\hat{\sigma}_\empintegral^2(\omega)}{2\numsim}
\Big).
\end{equation}
We shall test \eqref{eq_new_est} numerically; a statistical analysis of \eqref{eq_new_est} is beyond the scope of this paper.

\section{Numerical Experiments}\label{sec:numericalexperiments}
    
    In this final section, we present the simulation results. We consider two experiments: the first with only zero-mean colored noise, and the second with a deterministic signal embedded in colored noise. The results presented in this section were generated with GNU Octave 9.4.0. The corresponding code is available at \cite{code}.
    
	\subsection{Experiment 1: Zero-Mean noise.}
	
    The first experiment simulates the STFT of complex colored noise with PSD $\psd_{0.5,\, 4}$. We generate $\numsim=5$, $10$, $100$, and $1000$ realizations, apply the discrete STFT, and detect the zero set via the Minimal Grid Neighbors (MGN) method. We then assess the performance of $\hatrhoTF$ and $\widehat{\smoothedpsd{S}}(\omega)$, the empirical estimators for the first intensity and the smoothed PSD, respectively.
    	
    Figure~\ref{fig:exp1_spectrogram} displays the spectrogram on a logarithmic scale and the corresponding zeros for a single realization of this noise. We observe that close to $\frequency=0$, a high-energy ridge pushes the zeros slightly outward along the frequency axis. Conversely, a dense cloud of zeros accumulates around $\frequency=\pm 2$, where the spectrogram magnitude descends from approximately $-50$ to $-100$ dB.

    \begin{figure}[tbh]
		\centering
		\includegraphics[width=0.7\linewidth]{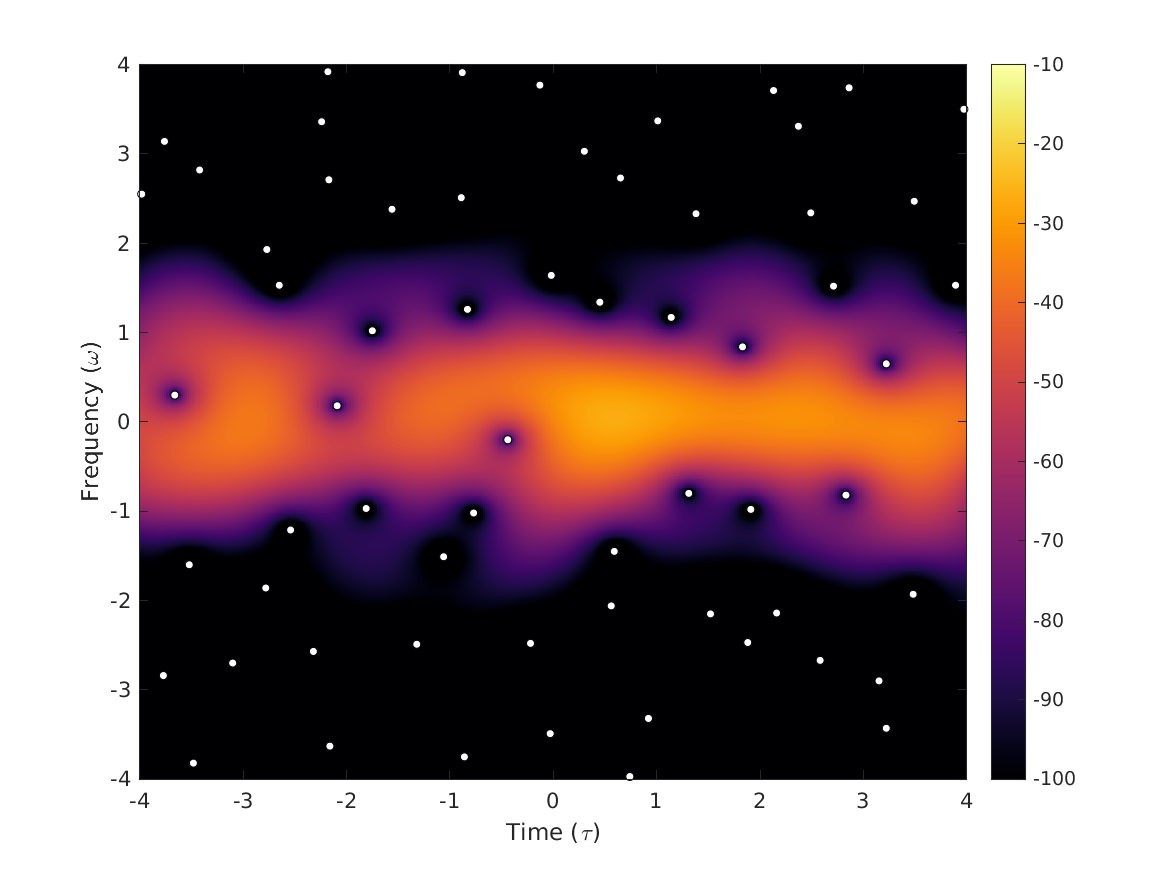}
		\caption{Spectrogram and zeros of colored noise with PSD $\psd_{0.5, 4}$ and zero mean.}
		\label{fig:exp1_spectrogram}
	\end{figure}

    The visual behavior of the zeros described above is quantitatively backed up by the first intensity $\rhoTF$ shown in Figure~\ref{fig:exp1_rho1_comparison}. Values strictly below $1$ are clearly seen around $\frequency=0$, confirming a lower spatial density of zeros in that region. In contrast, near $\frequency=\pm 2$, the first intensity rises above $1$, before ultimately approaching $1$ as the frequency moves further away from $0$, illustrating Prop.~\ref{prop:decay_polinomial}. Furthermore, the empirical estimate of the zero density, $\hatrhoTF(\omega)$, shows improved convergence toward the true theoretical distribution $\rhoTF(\omega)$ as $\numsim$ increases. This provides numerical support for both the accuracy of the discretization method and the first intensity formula derived in this work.

    \begin{figure}[tbh]
        \centering
        \hspace*{-0.7cm}\includegraphics[width=0.95\textwidth]{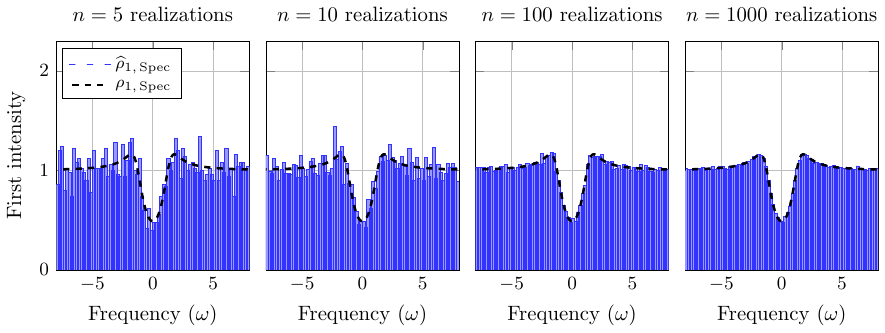}
        \caption{Comparison of the empirical zero density estimate $\hatrhoTF(\omega)$ as in \eqref{eq:histogram} for PSD $\psd_{0.5,\ 4}$ (zero mean) for various numbers of realizations vs. the true $\rhoTF(\omega)$.}
        \label{fig:exp1_rho1_comparison}
    \end{figure}
    	
    From the zero distribution, we estimate $\widehat{\smoothedpsd{S}}$ via \eqref{eq_new_est} (Figure~\ref{fig:exp1_smoothedpsd_comparison}). The estimate is accurate even with very few realizations, and converges to the true smoothed PSD $\smoothedpsd{S}$ as $\numsim$ grows. The perfect alignment at $\frequency=0$ is a direct consequence of the normalization conditions \eqref{eq:conditions_smooth_normalization}, as discussed in Remark~\ref{rem:identifiability}.
    	
	\begin{figure}[tbh]
		\centering
		\includegraphics[width=0.6\textwidth]{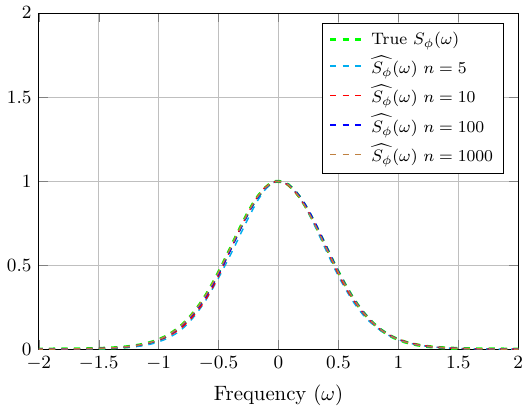}
		\caption{Estimation of the smoothed PSD $\smoothedpsd{S}$ as described in \eqref{eq_new_est} for different numbers of realizations.
		The signal is zero-mean colored noise with PSD $\psd_{0.5, 4}$.}
		\label{fig:exp1_smoothedpsd_comparison}
	\end{figure}
	
    \subsection{Experiment 2: Sensitivity to Non-Zero Mean}
	
	The second experiment investigates the performance of $\hatrhoTF$ and $\widehat{\smoothedpsd{S}}$ on signals with a non-zero mean. The rationale behind this experiment is that our estimators, although designed for zero-mean noise, remain effective provided the deterministic signal is sparse in the time-frequency plane. Specifically, if the signal's spectrogram manifests as thin, highly localized ridges, its high-energy content only alters the spatial distribution of the zeros locally. Because the background noise continues to dominate the vast majority of the plane, the zero-set statistics remain largely intact in regions where the signal perturbation is small. 
	
    To demonstrate this, we analyze a deterministic chirp, $f_1(\tau) = \exp(i \pi \tau^2)$, superimposed with colored noise having a PSD of $\psd_{1, 1}$. The Signal-to-Noise Ratio (SNR) is set to 20 dB.
	
    The spectrogram of this signal, shown in Figure~\ref{fig:exp2_spectrogram}, reveals how the high-energy ridge of the chirp acts as a repulsor, clearing zeros from its immediate vicinity. In the background, the decay of the noise PSD is also visible. The progression of the empirical zero density estimate, $\hatrhoTF(\omega)$, is shown in Figure~\ref{fig:exp2_rho1_comparison} for up to $\numsim=1000$ realizations, highlighting a slightly reduced density of zeros near $\frequency=0$, which coincides with the peak of the PSD. Figure~\ref{fig:exp2_smoothedpsd_comparison} shows the estimation of the smoothed PSD. Since our estimator is derived under a pure-noise assumption, the presence of the deterministic signal introduces a systematic bias. Consequently, the convergence is noticeably slower than in the zero-mean case, and even for a large number of realizations, the empirical curve provides a reasonable approximation rather than a perfect reconstruction of $\smoothedpsd{S}$.

	\begin{figure}[tbh]
		\centering
		\includegraphics[width=0.7\linewidth]{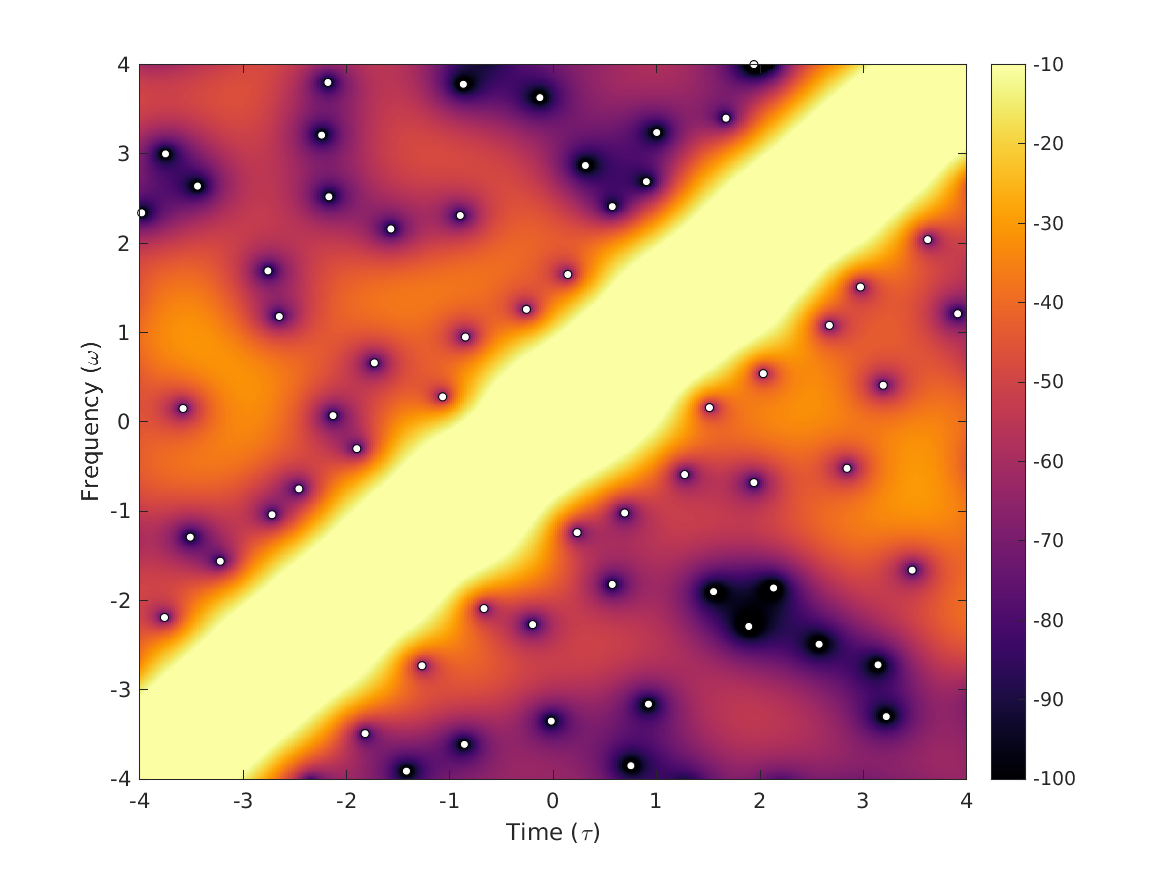}
		\caption{Spectrogram and zeros of a chirp $f_1(\tau)=\exp(i \pi \tau^2)$ with added complex colored Gaussian noise having PSD $\psd_{1,1}$ and an SNR of 20 dB.}
		\label{fig:exp2_spectrogram}
	\end{figure}

    \begin{figure}[tbh]
        \centering
        \hspace*{-0.8cm}\includegraphics[width=0.95\textwidth]{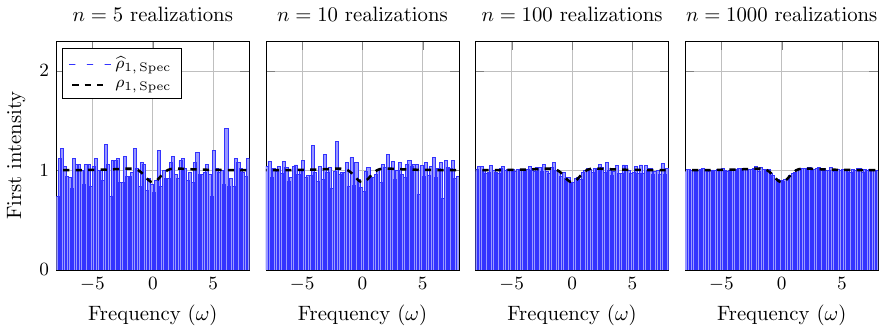}
        \caption{Comparison of the empirical zero density estimate $\hatrhoTF(\omega)$ for a chirp with added colored noise. As $\numsim$ increases, the estimate converges to the density of the background noise, even if the signal has a chirp.}
        \label{fig:exp2_rho1_comparison}
    \end{figure}

	\begin{figure}[tbh]
		\centering
		\includegraphics[width=0.6\textwidth]{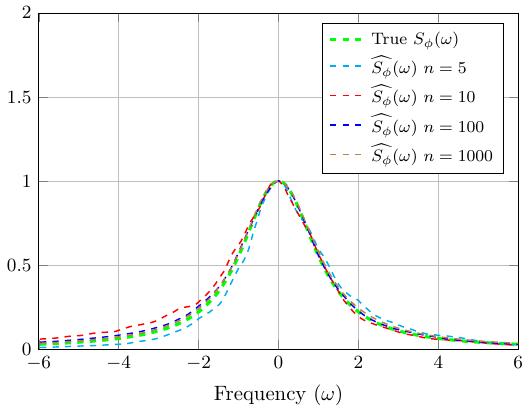}
		\caption{Estimation of the smoothed PSD ($\smoothedpsd{S}$) for the chirp signal with non-zero mean. The underlying noise parameters are successfully recovered.}
		\label{fig:exp2_smoothedpsd_comparison}
	\end{figure}

\end{document}